\documentclass{amsart}
\usepackage{amsmath, amsfonts, amssymb}
\usepackage{stmaryrd}
\usepackage{graphicx}
\usepackage{psfrag}
\usepackage{tableau}

\newtheorem{thm}{Theorem}[section] 
\newtheorem{lemma}[thm]{Lemma}
\newtheorem{prop}[thm]{Proposition}
\newtheorem{cor}[thm]{Corollary}
\newtheorem{conj}[thm]{Conjecture}

\theoremstyle{definition}
\newtheorem{remark}[thm]{Remark}
\newtheorem{example}[thm]{Example}
\newtheorem{defn}[thm]{Definition}

\DeclareMathOperator{\sh}{sh}
\DeclareMathOperator{\swap}{swap}
\DeclareMathOperator{\jdt}{jdt}
\DeclareMathOperator{\rect}{rect}
\DeclareMathOperator{\height}{height}
\DeclareMathOperator{\Gr}{Gr}
\DeclareMathOperator{\OG}{OG}
\DeclareMathOperator{\LG}{LG}
\DeclareMathOperator{\Span}{Span}
\DeclareMathOperator{\Ker}{Ker}
\DeclareMathOperator{\SO}{SO}
\DeclareMathOperator{\lis}{lis}
\DeclareMathOperator{\lds}{lds}
\def\Fl{\text{F$\ell$}}

\def\al{\alpha}
\def\alv{\alpha^\vee}
\def\be{\beta}

\def\ga{\gamma}
\def\gav{\gamma^\vee}

\def\la{\lambda}

\def\La{\Lambda}
\def\N{{\mathbb N}}
\def\Z{{\mathbb Z}}
\def\Q{{\mathbb Q}}
\def\C{{\mathbb C}}
\def\bP{{\mathbb P}}
\def\ba{{\mathbf a}}
\def\bb{{\mathbf b}}
\def\bc{{\mathbf c}}
\def\bd{{\mathbf d}}
\def\bu{{\mathbf u}}
\def\bv{{\mathbf v}}
\def\bw{{\mathbf w}}
\def\cO{{\mathcal O}}
\def\cP{{\mathcal P}}
\def\groth{{\mathfrak G}}
\def\ov{\overline}
\def\wh{\widehat}

\def\op{\mathrm{op}}
\def\equivB{{\,\,\wh\equiv\,\,}}

\def\ssm{\smallsetminus}
\def\noin{\noindent}

\newcommand{\pic}[2]{\includegraphics[scale=#1]{#2}}
\newcommand{\euler}[1]{\chi_{_{#1}}}
\newcommand{\ignore}[1]{}

\psfrag{1}{$1$}
\psfrag{2}{$2$}
\psfrag{3}{$3$}
\psfrag{4}{$4$}
\psfrag{5}{$5$}
\psfrag{6}{$6$}
\psfrag{7}{$7$}
\psfrag{dots}{$\cdots$}
\psfrag{n-2}{$n\!-\!2$}
\psfrag{n-1}{$n\!-\!1$}
\psfrag{n}{$n$}

\begin{document}

\title{$K$-theory of minuscule varieties}

\date{June 23, 2013}

\author{Anders Skovsted Buch}
\address{Department of Mathematics, Rutgers University, 110
  Frelinghuysen Road, Piscataway, NJ 08854, USA}
\email{asbuch@math.rutgers.edu\vspace{-2mm}}

\author{Matthew~J.~Samuel}
\email{msamuel@math.rutgers.edu}

\begin{abstract}
  Based on Thomas and Yong's $K$-theoretic jeu de taquin algorithm, we
  prove a uniform Littlewood-Richardson rule for the $K$-theoretic
  Schubert structure constants of all minuscule homogeneous spaces.
  Our formula is new in all types.  For the main examples of
  Grassmannians of type A and maximal orthogonal Grassmannians it has
  the advantage that the tableaux to be counted can be recognized
  without reference to the jeu de taquin algorithm.
\end{abstract}

\subjclass[2010]{Primary 14N15; Secondary 05E15, 19E08, 14M15}

\thanks{The first author was supported in part by NSF grants
  DMS-0906148 and DMS-1205351.}

\maketitle

\section{Introduction}\label{sec:intro}

The goal of this paper is to prove a uniform Littlewood-Richardson
rule for the $K$-theoretic Schubert structure constants of any
minuscule homogeneous space, also called a {\em minuscule variety}.
The family of minuscule varieties includes Grassmann varieties of type
A, maximal orthogonal Grassmannians, even dimensional quadric
hypersurfaces, as well as two exceptional varieties called the Cayley
plane and the Freudenthal variety.  The slightly larger family of {\em
  cominuscule varieties\/} also includes Lagrangian Grassmannians and
odd dimensional quadrics.  Several papers illustrate that many aspects
of the geometry and combinatorics of (co)minuscule varieties are
natural generalizations of Grassmannians of type A
\cite{proctor:bruhat, stembridge:fully, perrin:small*1,
  thomas.yong:combinatorial, chaput.manivel.ea:quantum*1}.

A result of Proctor \cite{proctor:bruhat} implies that the Schubert
varieties $X_\la$ in a cominuscule variety $X = G/P$ can be indexed by
order ideals $\la$ in a partially ordered set $\La_X$.  These order
ideals can be identified with shapes that generalize Young diagrams.
The Schubert structure sheaves $\cO_\la := [\cO_{X_\la}]$ form a
$\Z$-basis of the Grothendieck ring $K(X)$ of algebraic vector bundles
on $X$.  Brion has proved that the structure constants of $K(X)$ with
respect to this basis have signs that alternate with codimension
\cite{brion:positivity}.  Equivalently, there are unique non-negative
integers $c^\nu_{\la,\mu}$ for which the identity
\[
\cO_\la \cdot \cO_\mu = \sum_\nu (-1)^{|\nu|-|\la|-|\mu|}\,
c^\nu_{\la,\mu}\, \cO_\nu
\]
holds in $K(X)$, where $|\la|$ denotes the codimension of $X_\la$.
The structure constants $c^\nu_{\la,\mu}$ are generalizations of the
classical Littlewood-Richardson coefficients and describe the geometry
of intersections of Schubert varieties in $X$.

When $X$ is a Grassmann variety of type A, it was proved in
\cite{buch:littlewood-richardson} that the structure constant
$c^\nu_{\la,\mu}$ is equal to the number of certain combinatorial
objects called {\em set-valued tableaux}.  So far no generalization of
this rule has been found for other homogeneous spaces.  More recently,
Thomas and Yong have defined a $K$-theoretic version of
Sch\"utzenberger's {\em jeu de taquin\/} algorithm and conjectured
that, if $X$ is any minuscule variety, then the structure constant
$c^\nu_{\la,\mu}$ is equal to the number of increasing tableaux of
skew shape $\nu/\la$ that rectify to a {\em superstandard tableau\/}
$S_\mu$ by using $K$-theoretic jeu de taquin slides.  Thomas and Yong
proved their conjecture for Grassmannians of type A in
\cite{thomas.yong:jeu}, while the conjecture for maximal orthogonal
Grassmannians follows from combinatorial results of Clifford, Thomas,
and Yong \cite{clifford.thomas.ea:k-theoretic} together with a
$K$-theoretic Pieri formula proved by Buch and Ravikumar
\cite{buch.ravikumar:pieri}.  It turns out that the conjecture is also
true for the Cayley Plane and for even dimensional quadrics, but it
fails for the Freudenthal variety, see Remark~\ref{rmk:all_urt} and
Example~\ref{ex:e8_fails} below.  In this paper we show that, if the
superstandard tableau $S_\mu$ is replaced with a {\em minimal
  increasing tableau\/} $M_\mu$, then Thomas and Yong's conjecture
gives a correct formula for the $K$-theoretic Schubert structure
constants of all minuscule varieties.  We note that the possibility of
replacing the superstandard tableau $S_\mu$ with something else was
mentioned in \cite{thomas.yong:jeu}.

In contrast to Sch\"utzenberger's classical jeu de taquin algorithm,
the $K$-theoretic version operates with several empty boxes at the
same time.  Another important difference is that the $K$-theoretic
algorithm is not independent of choices, in the sense that different
choices of initial empty boxes may lead to several different
rectifications of the same skew tableau.  Thomas and Yong proved that
any superstandard tableau $S_\mu$ associated to a Grassmann variety of
type A has the property that, if any skew tableau $T$ has $S_\mu$ as a
rectification, then $S_\mu$ is the only rectification of $T$
\cite{thomas.yong:jeu}.  We will call such a tableau a {\em unique
  rectification target}.  The main combinatorial result of
\cite{clifford.thomas.ea:k-theoretic} states that superstandard
tableaux associated to maximal orthogonal Grassmannians are also
unique rectification targets.  

In this paper we attempt a systematic study of unique rectification
targets.  In particular, we prove that any minimal increasing tableau
associated to a minuscule variety is a unique rectification target.
We also prove several alternative versions of the
Littlewood-Richardson rule.  For example, if $T_0$ is {\em any\/}
increasing tableau of shape $\mu$, then we prove that
$c^\nu_{\la,\mu}$ is the number of increasing tableaux $T$ of shape
$\nu/\la$ for which the {\em greedy rectification\/} of $T$ is equal
to $T_0$.  This rectification is obtained by consistently choosing
{\em all\/} inner corners as initial empty boxes during the
rectification process.

When $X$ is a Grassmannian of type A, we also prove versions of the
Littlewood-Richardson rule that are formulated without reference to
the jeu de taquin algorithm.  To each permutation $w$ there is a {\em
  stable Grothendieck polynomial\/} $G_w$
\cite{lascoux.schutzenberger:structure, fomin.kirillov:grothendieck}
which can be identified with an element of the $K$-theory ring $K(X)$
\cite{buch:littlewood-richardson}.  For example, if $w$ is the
Grassmannian permutation associated to a partition $\mu$, then $G_w$
is equal to $\cO_\mu$.  It was proved in \cite{buch.kresch.ea:stable}
that the coefficient of $\cO_\nu$ in the expansion of $G_w$ is equal
to $(-1)^{|\nu|-\ell(w)}$ times the number of increasing tableaux $T$
of shape $\nu$ for which the {\em Hecke permutation\/} $w(T)$ is equal
to $w^{-1}$.  This permutation $w(T)$ is defined as the Hecke product
of the simple transpositions given by the row word of $T$.  We prove
more generally that the coefficient of $\cO_\nu$ in the expansion of
$\cO_\la \cdot G_w$ is equal to $(-1)^{|\nu/\la|-\ell(w)}$ times the
number of increasing tableaux $T$ of shape $\nu/\la$ such that $w(T) =
w^{-1}$.  By choosing $w$ to be the Grassmannian permutation for
$\mu$, this formula specializes to express the structure constants
$c^\nu_{\la,\mu}$.  Similar identities are obtained for maximal
orthogonal Grassmannians.

The main technical tool introduced in this paper is an equivalence
relation on words of integers that we call {\em $K$-Knuth
  equivalence}.  This relation is defined by using a mixture of the
basic relations defining the Hecke monoid and the plactic algebra.  We
prove that if $T$ and $T'$ are increasing tableaux defined on skew
Young diagrams, then $T'$ can be obtained from $T$ by a sequence of
forward and reverse $K$-theoretic jeu de taquin slides if and only if
the the row words of $T$ and $T'$ are $K$-Knuth equivalent.  This
result implies that the Hecke permutation $w(T)$ is an invariant under
jeu de taquin slides.  It also implies that the length of the longest
strictly increasing subsequence of the row word of a tableau is an
invariant, as proved earlier by Thomas and Yong
\cite{thomas.yong:jeu}.

Our paper is organized as follows.  Section~\ref{sec:comin} explains
Proctor's bijection between order ideals and Schubert classes and
gives self-contained proofs of its main properties.
Section~\ref{sec:kjdt} explains Thomas and Yong's $K$-theoretic jeu de
taquin algorithm and gives several examples.  We also define a {\em
  combinatorial $K$-theory ring\/} associated to any partially ordered
set that satisfies certainly properties.  In Section~\ref{sec:ktheory}
we then show that, if $X$ is any minuscule variety, then the
Grothendieck ring $K(X)$ is isomorphic to the combinatorial $K$-theory
ring associated to $\La_X$.  We then use the geometry of $X$ to
establish additional combinatorial facts, including a criterion for
unique rectification targets and new versions of the
Littlewood-Richardson rule.  Section~\ref{sec:kknuth} defines the
$K$-Knuth equivalence relation, as well as {\em resolutions\/} of
increasing tableaux with empty boxes, and {\em reading words\/} of
these resolutions.  The main result in this section states that all
resolutions of a tableau with empty boxes have $K$-Knuth equivalent
reading words.  Section~\ref{sec:grass} is devoted to Grassmannians of
type A.  We first use the $K$-Knuth equivalence of resolutions to
establish that the $K$-Knuth class of the row word of an increasing
tableau is an invariant under jeu de taquin slides.  We apply this to
prove that minimal increasing tableaux are unique rectification
targets, as well as a criterion that generalizes Thomas and Yong's
result that superstandard tableaux are unique rectification targets.
We also prove the above mentioned formula for products involving a
stable Grothendieck polynomial.  Finally, we prove that the
rectification of a minimal increasing tableau of skew shape is always
a minimal increasing tableau.  This fact implies that {\em maximal
  increasing tableaux\/} are also unique rectification targets, which
in turn is used in the proof of the Littlewood-Richardson rule based
on greedy rectifications.  Section~\ref{sec:maxog} finally gives
analogues of the results of Section~\ref{sec:grass} for maximal
orthogonal Grassmannians.  While many of these results can be
translated from type A, we also define a weak version of the $K$-Knuth
relation that governs the jeu de taquin algorithm for maximal
orthogonal Grassmannians.

We thank Robert Proctor and John Stembridge for answers to questions
about the history of minuscule posets.  We also thank Alex Yong for
inspiring discussions.


\section{Schubert varieties in cominuscule spaces}
\label{sec:comin}

Let $X = G/P$ be a homogeneous space defined by a semisimple complex
linear algebraic group $G$ and a parabolic subgroup $P$.  Fix also a
maximal torus $T$ and a Borel subgroup $B$ such that $T \subset B
\subset P \subset G$.  Let $R$ be the corresponding root system, with
positive roots $R^+$ and simple roots $\Delta \subset R^+$.  Let $W =
N_G(T)/T$ be the Weyl group of $G$ and $W_P = N_P(T)/T \subset W$ the
Weyl group of $P$.  Each element $w \in W$ defines a Schubert variety
$X(w) = \ov{Bw.P} \subset X$, which depends only on the coset of $w$
in $W/W_P$.  Let $W^P \subset W$ be the set of minimal length
representatives for the cosets in $W/W_P$.  An element $w \in W$
belongs to $W^P$ if and only if $w.\be \in R^+$ for each $\be \in
\Delta$ with $s_\be \in W_P$.  The set $W^P$ is in one-to-one
correspondence with the $B$-stable Schubert varieties in $X$, and for
$w \in W^P$ we have $\dim X(w) = \ell(w)$.

Assume that $P$ is a maximal parabolic subgroup of $G$.  Then $P$
corresponds to a unique simple root $\gamma \in \Delta$ such that
$W_P$ is generated by all simple reflections except $s_\ga$.  Given
any root $\al \in R$, let $\ga(\al)$ denote the coefficient of $\ga$
when $\al$ is written as a linear combination of simple roots.  The
simple root $\ga$ and the variety $X$ are called {\em cominuscule\/}
if $|\ga(\al)| \leq 1$ for all $\al \in R$, and $\ga$ and $X$ are
called {\em minuscule\/} if $\ga^\vee$ is a cominuscule simple root in
the dual root system $R^\vee = \{ \alv \mid \al \in R \}$.  Here $\alv
= \frac{2\,\al}{(\al,\al)}$ denotes the coroot of $\al$.

The collection of all minuscule and cominuscule varieties is listed in
Table~\ref{tab:comin}.  Notice that all minuscule varieties are also
cominuscule, possibly for a different Lie type.  For example, the
minuscule odd orthogonal Grassmannian $\OG(n,2n+1)$ of type $B_n$ is
isomorphic to the even orthogonal Grassmannian $\OG(n+1,2n+2)$ of type
$D_{n+1}$ which is both minuscule and cominuscule.  In this section we
assume that $\ga$ is a cominuscule simple root.

\begin{table}[!ht]
\caption{Minuscule and cominuscule varieties}
\centering
\begin{tabular}{|l|l|}
\hline
\multicolumn{2}{|l|}{\vspace{-3mm}}\\
\multicolumn{2}{|l|}{\hspace{20mm}\pic{.2}{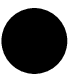} \ \ Minuscule simple root. }\\
\multicolumn{2}{|l|}{\hspace{20mm}\pic{.2}{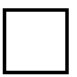} \ \ Cominuscule simple root. }\\
\multicolumn{2}{|l|}{\hspace{20mm}\pic{.2}{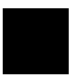} \ \ Minuscule and cominuscule simple root. }\\
\multicolumn{2}{|l|}{\vspace{-3mm}}\\
\hline
&\vspace{-3mm}\\
\ Type $A_n$: & \ Type $B_n$: \\ 
&\vspace{-1mm}\\
\ \pic{.2}{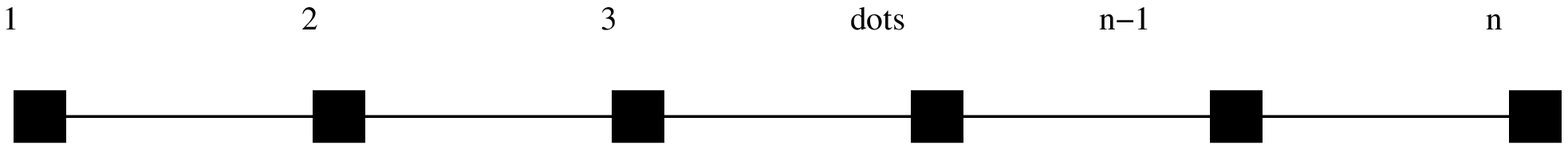} & \ \pic{.2}{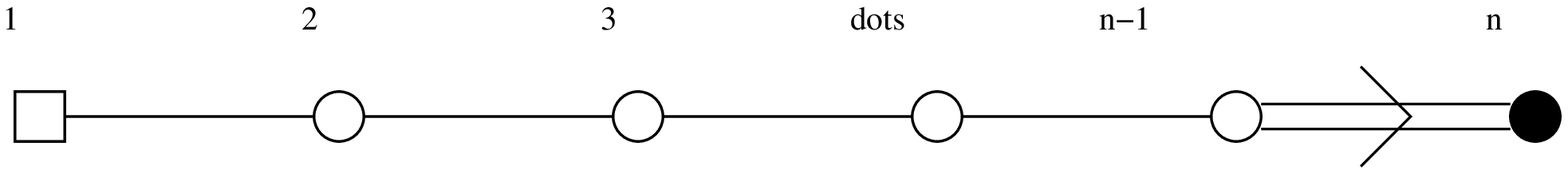} \\
&\vspace{-3mm}\\
\ $A_n/P_m = \Gr(m,n+1)$ & \ $B_n/P_1 = Q^{2n-1}$ \ Odd quadric. \\
\ Grassmannian of type A. &\vspace{-2mm}\\
& \ $B_n/P_n = \OG(n,2n+1)$ \\
& \ Max.\ orthogonal Grassmannian. \\
&\vspace{-3mm}\\
\hline
&\vspace{-3mm}\\
\ Type $C_n$: & \ Type $D_n$: \\
&\vspace{-1mm}\\
\ \pic{.2}{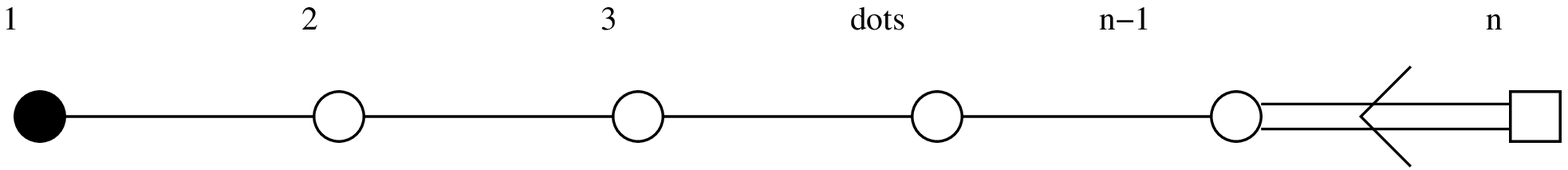} &\\
&\vspace{-3mm}\\
\ $C_n/P_1 = \bP^{2n-1}$ \ Projective space. \ &\\
&\vspace{-2mm}\\
\ $C_n/P_n = \LG(n,2n)$ &\\
\ Lagrangian Grassmannian. &\vspace{-23mm}\\
& \ \pic{.2}{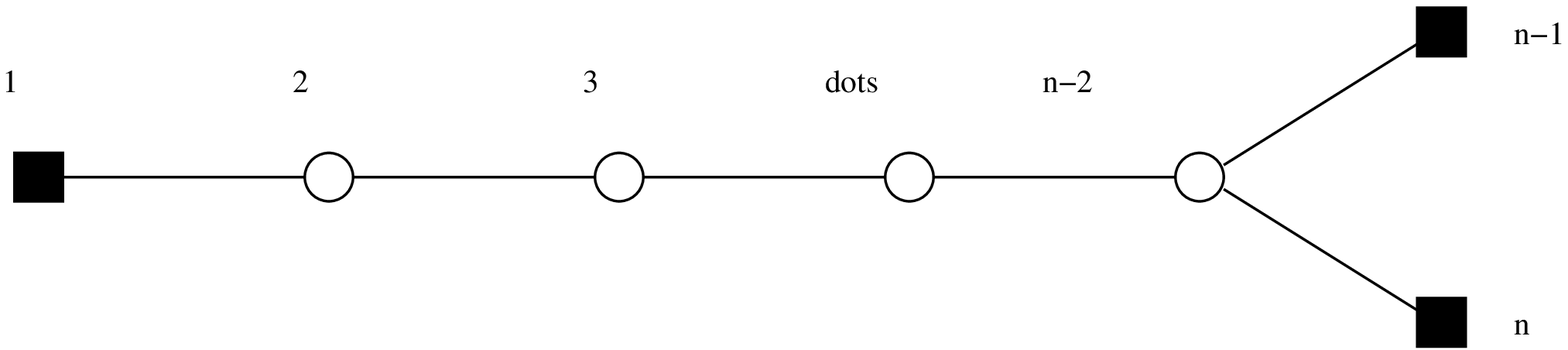} \\
&\vspace{-3mm}\\
& \ $D_n/P_1 = Q^{2n-2}$ \ Even quadric. \\
&\vspace{-2mm}\\
& \ $D_n/P_{n-1} \cong D_n/P_n = \OG(n,2n)$ \ \\
& \ Max.\ orthogonal Grassmannian. \\
&\vspace{-3mm}\\
\hline
&\vspace{-3mm}\\
\ Type $E_6$: & \ Type $E_7$: \\
&\vspace{-1mm}\\
\ \pic{.2}{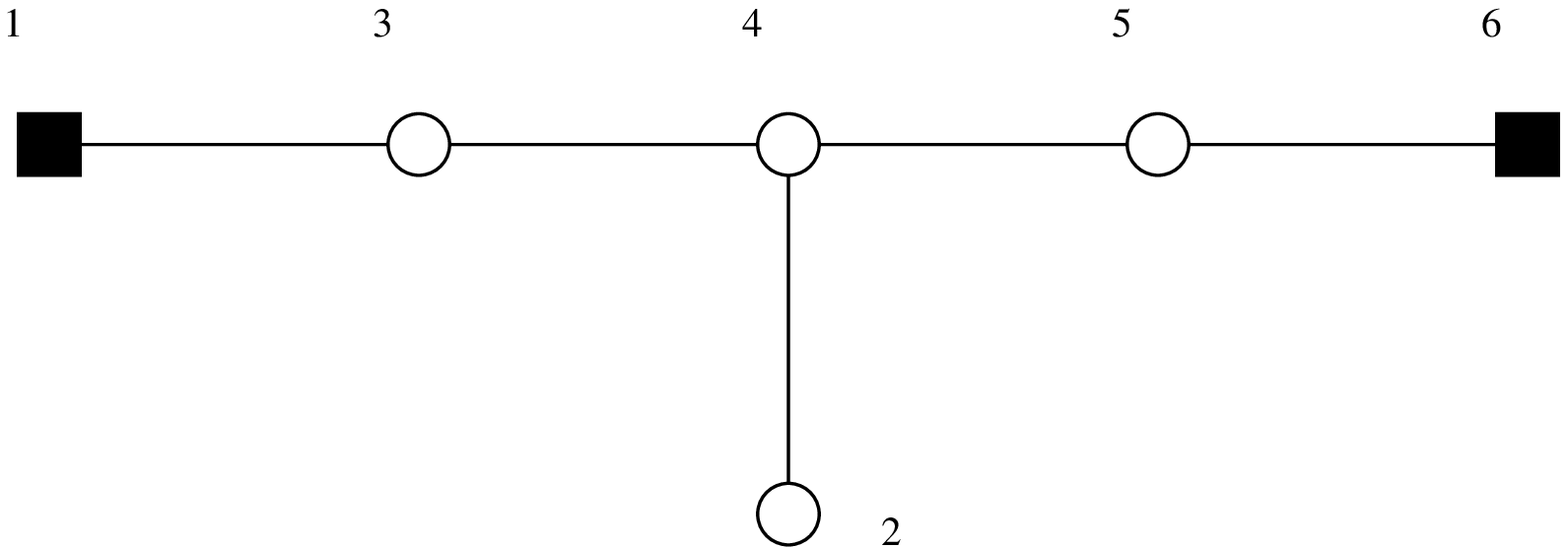} & \ \pic{.2}{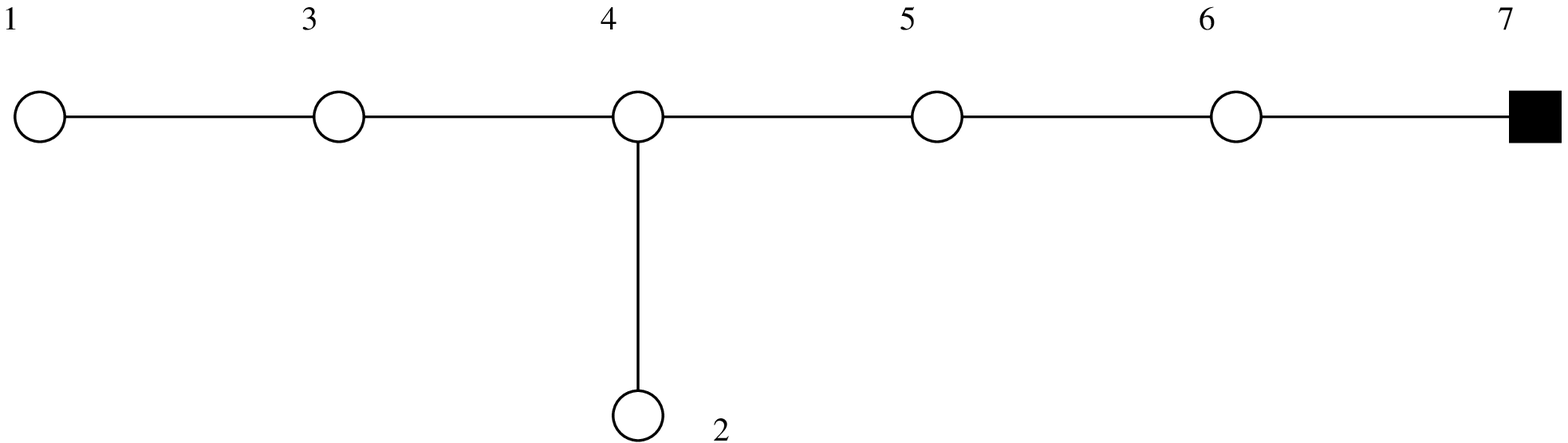} \\
&\vspace{-3mm}\\
\ $E_6/P_1 \cong E_6/P_6$ \ Cayley plane. & \ $E_7/P_7$ \ Freudenthal variety.
\vspace{-3mm}\\
&\\
\hline
\end{tabular}
\label{tab:comin}
\end{table}


A result of Proctor \cite{proctor:bruhat} shows that the Bruhat order
on $W^P$ is a distributive lattice.  If we let $\La_X$ denote the
partially ordered set of join irreducible elements of this lattice,
then Birkhoff's representation theorem gives a bijection between $W^P$
and the set of order ideals in $\La_P$.  Proctor identified $W^P$ with
the set of weights of a minuscule representation, and in this way
$\La_X$ is identified with a subset of the positive roots $R^+$.
Stembridge later gave a different construction of $\La_X$ as the heap
of the longest element of $W^P$ \cite{stembridge:fully}.  In this
expository section we start by defining the set $\La_X$ using the
description given in \cite[Thm.~11]{proctor:bruhat}, then prove
directly from the cominuscule condition on $\ga$ that the order ideals
of $\La_X$ are in one-to-one correspondence with the Schubert classes
of $X$.  An alternative treatment can be found in \cite[\S
2]{thomas.yong:combinatorial}.  Let $\leq$ denote the partial order on
$\Span_\Z(\Delta)$ defined by $\al \leq \be$ if and only if $\be-\al$
is a sum of simple roots.

\begin{defn}\label{def:shape}
  Let $\La_X = \{ \al \in R \mid \ga(\al)=1 \}$ be the set of positive
  roots for which the coefficient of $\ga$ is non-zero, and equip this
  set with the restriction of the partial order $\leq$ on
  $\Span_\Z(\Delta)$.  A subset $\la \subset \La_X$ is called a {\em
    straight shape\/} if it is a lower order ideal, i.e.\ if $\al \in
  \la$ then $\la$ also contains all roots $\be \in \La_X$ for which
  $\be \leq \al$.  A {\em skew shape\/} is any set theoretic
  difference $\la/\mu := \la \ssm \mu$ of straight shapes $\la$ and
  $\mu$; this notation will be used only when $\mu \subset \la$.
\end{defn}

The partially ordered set $\La_X$ can be identified with a subset of
$\N^2$, where the order on $\N^2$ is defined by $[r_1,c_1] \leq
[r_2,c_2]$ if and only if $r_1\leq r_2$ and $c_1\leq c_2$.  We will
represent $\N^2$ as a grid of boxes $[r,c]$, where the first
coordinate is a row number (increasing from top to bottom) and the
second coordinate is a column number (increasing from left to right).
The shape of the resulting set of boxes $\La_X$ is displayed in
Table~\ref{tab:shapes}.  We will henceforth call the roots in
$\Lambda_X$ for {\em boxes}.

\begin{table}
  \caption{The shape $\La_X$ of a cominuscule variety $X$}
  \begin{tabular}{|l|l|}
    \hline
    &\vspace{-3mm}\\
    \ Grassmannian of type A & 
    \ Lagrangian Grassmannian \\
    &\vspace{-3mm}\\
    \ $\La_{\Gr(m,m+k)} \ = \ 
    \tableau{10}{
      {}&{}&{}&{}&{} \\
      {}&{}&{}&{}&{} \\
      {}&{}&{}&{}&{} \\
      {}&{}&{}&{}&{}
    }$
    &
    \ $\La_{\LG(n,2n)} \ = \ 
    \tableau{10}{
      {}&{}&{}&{}&{} \\
      &{}&{}&{}&{} \\
      &  &{}&{}&{} \\
      &  &  &{}&{} \\
      &  &  &  &{}
    }$
    \\
    &\vspace{-3mm}\\
    \mbox{}\hspace{13mm}($m$ rows and $k$ columns) & 
    \mbox{}\hspace{28mm}($n$ rows) \\
    &\vspace{-2mm}\\
    \hline
    &\vspace{-3mm}\\
    \ Odd quadric & 
    \ Max.\ orthogonal Grassmannian \\
    &\vspace{-3mm}\\
    \ $\La_{Q^{2n-1}} \ = \ 
    \tableau{10}{
      {}&{}&{}&{}&{}&{}&{}&{}&{}
    }$ \ 
    &
    \ $\La_{\OG(n,2n)} \ = \ 
    \tableau{10}{
      &{}&{}&{}&{} \\
      &  &{}&{}&{} \\
      &  &  &{}&{} \\
      & & & &{}
    }$ \\
    &\vspace{-3mm}\\
    & \mbox{}\hspace{26mm}($n-1$ rows) \vspace{-9mm}\\
    \mbox{}\hspace{25mm}($2n-1$ boxes) & \vspace{3mm}\\
    &\vspace{-1mm}\\
    \hline
    &\vspace{-3mm}\\
    \ Even quadric, $n$ even & 
    \ Even quadric, $n$ odd \\
    &\vspace{-3mm}\\
    \ $\La_{Q^{2n}} \ = \
    \tableau{10}{
      {}&{}&{}&{} \\
      & &{}&{}&{}&{} 
    }$
    &
    \ $\La_{Q^{2n}} \ = \ 
    \tableau{10}{
      {}&{}&{}&{}&{} \\
      &  &  &{}&{} \\
      &  &  &  &{} \\
      &  &  &  &{} \\
      &  &  &  &{}
    }$
    \\
    &\vspace{-2mm}\\
    & \mbox{}\hspace{18mm}($2n$ boxes) \vspace{-10mm}\\
    \mbox{}\hspace{15mm}($n$ boxes in each row) & \vspace{4mm}\\
    &\vspace{-1mm}\\
    \hline
    &\vspace{-3mm}\\
    \ Cayley plane &
    \ Freudenthal variety \\
    &\vspace{-3mm}\\
    \ $\La_{E_6/P_6} \ = \ 
    \tableau{10}{
      {}&{}&{}&{}\\
      &  &{}&{}&{}&{}\\
      &  &{}&{}&{}&{}\\
      & & & &{}&{}&{}&{}
    }$
    &
    \ $\La_{E_7/P_7} \ = \ 
    \tableau{10}{
      {}&{}&{}&{}&{}\\
      &  &  &{}&{}&{}&{}&{}\\
      &  &  &{}&{}&{}&{}&{}\\
      &  &  &  &  &{}&{}&{}\\
      &  &  &  &  &  &{}&{}&{}\\
      &  &  &  &  &  &{}&{}&{}\\
      &  &  &  &  &  &  &  &{}\\
      &  &  &  &  &  &  &  &{}\\
      & & & & & & & &{}
    }$ \ 
    \vspace{-1mm}\\
    &\\
    \hline
  \end{tabular}
  \label{tab:shapes}
\end{table}


We remark that there are other ways to identify $\La_X$ with a subset
of $\N^2$, and the choices made in Table~\ref{tab:shapes} are slightly
different from those made in \cite{thomas.yong:combinatorial}.  The
shapes chosen here satisfy that the longest element $w_X$ in $W_P$
acts on $\La_X$ as a 180 degree rotation for shapes in the left column
of Table~\ref{tab:shapes}, while $w_X$ acts by reflection in a
south-west to north-east diagonal for shapes in the right column of
the table.  The action of $w_X$ is related to Poincare duality, see
Corollary~\ref{cor:poincare} below.  We have also chosen the shapes so
that the number of boxes in each row decreases from top to bottom.  As
a result, if $\la \subset \La_X$ is any straight shape, then $\la$ can
be represented by the {\em partition\/} $(\la_1 \geq \la_2 \geq \dots
\geq \la_\ell)$, where $\la_r$ is the number of boxes in row $r$ of
$\la$.

One can check that for each box $\al \in \La_X$ with $\al > \ga$,
there exists a simple root $\be \in \Delta \ssm \{\ga\}$ such that
$\al-\be \in \La_X$.  This implies that $\height(\al)$, defined as the
sum of the coefficients obtained when $\al$ is written as a linear
combination of simple roots, is equal to the maximal cardinality of a
totally ordered subset of $\La_X$ with $\al$ as its maximal element.

\begin{lemma}\label{lem:incomp_box}
  Let $\al, \be \in \La_X$.  If $\al \not\leq \be$ and $\be \not\leq
  \al$, then $(\al,\be)=0$.
\end{lemma}
\begin{proof}
  The cominuscule condition implies that $\al + \be \notin R$, and the
  incomparability gives $\al - \be \notin R$.  The lemma therefore
  follows from \cite[Lemma~9.4]{humphreys:introduction}.
\end{proof}

\begin{defn}\label{def:shape2wge}
  Given a straight shape $\la \subset \La_X$, define an element $w_\la
  \in W$ as follows.  If $\la = \emptyset$, then set $w_\la = 1$.
  Otherwise set $w_\la = w_{\la\ssm\{\al\}}\, s_\al \in W$, where $\al
  \in \la$ is any maximal box.
\end{defn}

To see that $w_\la$ is well defined, assume that $\al$ and $\be$ are
two distinct maximal boxes of the straight shape $\la$.  Then
Lemma~\ref{lem:incomp_box} implies that $(\al,\be)=0$, so $s_\al$
commutes with $s_\be$.  By induction on $|\la|$ we may assume that
$w_{\la\ssm\{\al\}} = w_{\la\ssm\{\al,\be\}}\,s_\be$ and
$w_{\la\ssm\{\be\}} = w_{\la\ssm\{\al,\be\}}\,s_\al$ are well defined,
so we obtain $w_{\la\ssm\{\al\}}\, s_\al = w_{\la\ssm\{\al,\be\}}\,
s_\be\, s_\al = w_{\la\ssm\{\al,\be\}}\, s_\al\, s_\be =
w_{\la\ssm\{\be\}}\, s_\be$, as required.

For any element $w \in W$, let $I(w) = \{ \al \in R^+ \mid w.\al < 0
\}$ denote the inversion set of $w$.  This set uniquely determines $w$
\cite[Thm.~10.3]{humphreys:introduction}, and we have $\ell(w) =
|I(w)|$.  The following proposition shows that the map from $W^P$ to
straight shapes in $\La_X$ is given by $w \mapsto I(w)$, see also
\cite[Prop.~2.1]{thomas.yong:combinatorial}.

\begin{thm}\label{thm:shapes}
  If $\la \subset \La_X$ is a straight shape, then $w_\la \in W^P$ and
  $I(w_\la) = \la$.  On the other hand, if $u \in W^P$, then $I(u)$ is
  a straight shape in $\La_X$ and $w_{I(u)} = u$.
\end{thm}
\begin{proof}
  The first claim is clear if $\la = \emptyset$, so assume that
  $\la\neq\emptyset$ and let $\al\in\la$ be a maximal box.  Then
  $w_\la = w_{\la\ssm\{\al\}}\, s_\al$, and by induction on $|\la|$ we
  have $I(w_{\la\ssm\{\al\}}) = \la\ssm\{\al\}$.  Let $\be \in R^+$ be
  any positive root.  We claim that $w_\la.\be < 0$ if and only if
  $\be \in \la$.  We have $s_\al.\be = \be - (\be,\alv) \al$, so the
  cominuscule condition implies that $\ga(\be) - (\be,\alv) =
  \ga(s_\al.\be) \in \{0,\pm 1\}$.  The claim is clear if
  $(\be,\alv)=0$, so assume that $(\be,\alv)\neq 0$.  If $\ga(\be)=0$,
  then we must have $(\be,\alv) = \pm 1$, and the claim is true
  because we have either $s_\al.\be = \be+\al > \al$ or $0 <
  -s_\al.\be = \al-\be < \al$.  Otherwise we have $\ga(\be)=1$, and
  the cominuscule condition implies that $(\be,\alv) \in \{1,2\}$.  In
  this case \cite[Lemma~9.4]{humphreys:introduction} implies that $\be
  \leq \al$ or $\be > \al$.  If $\be > \al$, then the claim is true
  because we have either $s_\al.\be = \be-\al > 0$ and
  $\ga(s_\al.\be)=0$, or $0 < -s_\al.\be = 2\al-\be < \al$.  Finally,
  if $\be \leq \al$ then the claim follows because we have either
  $-s_\al.\be = \al-\be > 0$ and $\ga(s_\al.\be)=0$, or $-s_\al.\be =
  2\al - \be \geq \al$.  It follows from the claim that $I(w_\la) =
  \la$, which in turn implies that $w_\la \in W^P$.

  Let $u \in W^P$.  Since $u.(\Delta \ssm \{\gamma\}) \subset R^+$ we
  get $u.(R^+ \ssm \Lambda_X) \subset R^+$.  It follows that $I(u)
  \subset \Lambda_X$.  To see that $I(u)$ is a straight shape in
  $\Lambda_X$, let $\al \in I(u)$, $\be \in \Lambda_X$, and assume
  that $\be \leq \al$.  Then $\al - \be$ is a sum of simple roots from
  the set $\Delta \ssm \{\gamma\}$, so $u.(\al - \be) \geq 0$, hence
  $u.\be \leq u.\al < 0$ and $\be \in I(u)$.  Finally, the identity
  $w_{I(u)} = u$ follows because $I(w_{I(u)}) = I(u)$.
\end{proof}

\begin{remark}\label{rmk:coroot_poset}
  If the simple root $\ga$ is minuscule but not cominuscule, then the
  Schubert varieties in $X$ are given by straight shapes in the
  partially ordered set $\La_X := \{ \alv \in R^\vee \mid \gav(\alv)=1
  \}$, consisting of positive coroots for which the coefficient of the
  cominuscule coroot $\gav$ is non-zero.  Alternatively, one can
  construct $X$ using a group $G$ of a different Lie type, so that
  $\ga$ is both minuscule and cominuscule.
\end{remark}

The Bruhat order on $W^P$ is defined by $u \leq v$ if and only if
$X(u) \subset X(v)$.

\begin{cor}\label{cor:bruhat}
  Let $\la$ and $\mu$ be straight shapes in $\La_X$.  Then $w_\mu \leq
  w_\la$ if and only if $\mu \subset \la$.
\end{cor}
\begin{proof}
  We may assume that $|\la| = |\mu|+1$, in which case $w_\mu \leq
  w_\la$ if and only if $w_\mu^{-1} w_\la$ is a reflection \cite[\S
  5.9]{humphreys:reflection}.  If $\mu \subset \la$, then this is true
  because $w_\la = w_\mu\, s_\al$ where $\{\al\} = \la \ssm \mu$.
  Assume that $w_\mu \leq w_\la$ and choose $\al \in R^+$ such that
  $w_\la = w_\mu\, s_\al$.  It follows from
  \cite[Prop.~5.7]{humphreys:reflection} that $\al \in \la\ssm\mu$.
  We must show that $w_\la.\be < 0$ for each $\be \in \mu$.  This is
  clear if $(\be,\alv)=0$, so assume that $(\be,\alv)\neq 0$.  Since
  $\gamma(s_\al.\be) = 1 - (\be,\alv) \in \{0,\pm 1\}$, we have
  $(\be,\alv) > 0$, and we deduce from
  \cite[Lemma~9.4]{humphreys:introduction} that $\be < \al$.  This
  implies that $\be \in \la$, as required.
\end{proof}

Let $w_X$ be the longest element of $W_P$.  Then $w_X.\La_X = \La_X$.
In fact, since $w_X.(\Delta \ssm \{\ga\}) \subset -R^+$, it follows
that $w_X$ is an order-reversing involution of $\La_X$.  As mentioned
above, $w_X$ acts as a rotation on the shapes in the left column of
Table~\ref{tab:shapes} and as a reflection on the shapes in the right
column.  If $\la \subset \La_X$ is a straight shape, then $w_X.\la$ is
an upper order ideal of $\La_X$.  The {\em Poincare dual shape\/} of
$\la$ is the straight shape $\la^\vee = \La_X \ssm w_X.\la$.

\begin{example}
  Let $X = E_6/P_6$ be the Cayley plane.  Then the straight shape $\la
  = (4,2,1)$ has Poincare dual shape $\la^\vee = (4,3,2)$.
  \[
  \la = \tableau{10}{[A]{}&{}&{}&{}\\ &&{}&{}&[lt]{}&[tr]{}\\ 
  &&[A]{}&&&[r]\\ &&&[t]&[lb]&[b]&[h]&[hr]}
  \hspace{20mm}
  \la^\vee = \tableau{10}{[A]{}&{}&{}&{}\\ &&{}&{}&{}&[tr]\\ 
  &&[A]{}&{}&&[r]\\ &&&&[lb]&[b]&[h]&[hr]}
  \]
\end{example}

\begin{cor}\label{cor:poincare}
  Let $\la \subset \La_X$ be a straight shape.  Then $w_{\la^\vee} =
  w_0 w_\la w_X$ is the Poincare dual Weyl group element of $w_\la$,
  where $w_0$ is the longest element in $W$.
\end{cor}
\begin{proof}
  If $\be \in R^+$ is any positive root, then we have $w_0 w_\la
  w_X.\be < 0$ if and only if $w_X.\be \in R^+ \ssm \la$, and since
  $w_X.(\Delta \ssm \{\ga\}) \subset -R^+$, this holds if and only if
  $w_X.\be \in \La_X\ssm\la = w_X.\la^\vee$.  It follows that $I(w_0
  w_\la w_X) = \la^\vee$, as required.
\end{proof}

To each straight shape $\la \subset \La_X$ we assign the Schubert
variety $X_\la := X(w_{\la^\vee})$ of codimension $|\la|$ in $X$.  If
$\mu$ is an additional straight shape, then Corollary~\ref{cor:bruhat}
implies that $X_\la \subset X_\mu$ if and only if $\mu \subset \la$.
For the classical Grassmannians $\Gr(m,n)$, $\LG(n,2n)$, and
$\OG(n,2n)$ that parametrize subspaces of a vector space, one can also
define the Schubert variety corresponding to a partition $\la$ by
using incidence conditions relative to a fixed flag of subspaces (see
e.g.\ \cite{buch.ravikumar:pieri} for these standard constructions).
Since the Bruhat order is also determined by containment of partitions
in these constructions, one may deduce from the following lemma that
the standard constructions agree with the assignment $\la \mapsto
X_\la$ used here.  In particular, the $K$-theoretic Pieri formulas
proved in \cite{lenart:combinatorial, buch.ravikumar:pieri} are valid
with the notation used here.

\begin{lemma}\label{lem:bruhat_auto}
  Any automorphism of the set $W^P$ that preserves the Bruhat order
  arises from an automorphism of the Dynkin diagram of $\Delta$ that
  fixes $\ga$.
\end{lemma}
\begin{proof}
  An order preserving automorphism of $W^P$ is equivalent to an
  inclusion preserving automorphism of the set of straight shapes in
  $\La_X$.  Since such an automorphism must restrict to an
  automorphism of the straight shapes that contain a unique maximal
  box, it must be given by an automorphism of the partially ordered
  set $\La_X$.  If $X$ is the Cayley plane or the Freudenthal variety,
  then we leave it as an exercise to check that the only automorphism
  of $\La_X$ is the identity.  
  
  We will say that a box $\al \in \La_X$ is {\em extreme\/} if $\al$
  belongs to a unique maximal totally ordered subset of $\La_X$.  If
  $X$ is not the Cayley plane or the Freudenthal variety, then an
  inspection of Table~\ref{tab:shapes} shows that $\La_X$ contains at
  least one extreme box.  For example, if $X = \Gr(m,n)$ is a
  Grassmannian of type A, then the upper-right box and the lower-left
  box of $\La_X$ are extreme boxes.  Furthermore, $\La_X$ contains two
  distinct extreme boxes of the same height if and only if $X =
  \Gr(n,2n)$ or $X = Q^{2n}$ for some $n \geq 2$.  These are also the
  cases where the Dynkin diagram has a nontrivial automorphism that
  fixes $\ga$, and this automorphism defines an automorphism $\iota$
  of $\La_X$ that interchanges the two extreme boxes.
  
  Let $\psi$ be any automorphism of $\La_X$.  Then $\psi$ maps each
  extreme box $\al$ to an extreme box $\al'$ of the same height.  If
  $\al' \neq \al$, then we can replace $\psi$ with $\psi\iota$ to
  obtain that $\psi$ maps every extreme box to itself.  It is enough
  to show that $\psi$ is the identity.  Let $\al \in \La_X$ be an
  extreme box and let $P \subset \La_X$ be the unique maximal totally
  ordered subset containing $\al$.  Then $\psi$ is the identity on
  $P$, so $\psi$ restricts to an automorphism of $\La_X \ssm P$.  This
  partially ordered set is isomorphic to $\La_{X'}$ for a cominuscule
  variety $X'$ of smaller dimension.  Since $\psi$ fixes at least one
  extreme box of $\La_X \ssm P$, it follows by induction that $\psi$
  is the identity on $\La_X \ssm P$, as required.
\end{proof}

The methods of this section can be used to prove that every element
$w \in W^P$ is {\em fully commutative}, i.e.\ any reduced expression
for $w$ can be obtained from any other by interchanging commuting
simple reflections.  This was proved by Fan for simply laced root
systems \cite{fan:hecke} and by Stembridge for non-simply laced root
systems \cite{stembridge:fully}.  The remainder of this section will
not be used in the rest of our paper.

Let $\mu \subset \la$ be straight shapes in $\La_X$.  Then
Definition~\ref{def:shape2wge} implies that $w_\mu^{-1} w_\la$ is the
product of all reflections $s_\al$ for $\al \in \la/\mu$, in any order
compatible with the partial order $\leq$ on $\La_X$.  However, the
length of this product is hard to predict.

On the other hand, it follows from Theorem~\ref{thm:shapes} that
$I(w_\la w_\mu^{-1}) = w_\mu.(\la/\mu)$, so $\ell(w_\la w_\mu^{-1}) =
|\la/\mu|$.  Furthermore, the product $w_\la w_\mu^{-1}$ depends only
on the skew shape $\la/\mu$.  Indeed, if $\al \in \La_X\ssm\la$ is any
box such that $\mu\cup\{\al\}$ is a straight shape, then
$w_{\la\cup\{\al\}}\, {w_{\mu\cup\{\al\}}}^{-1} = w_\la s_\al (w_\mu
s_\al)^{-1} = w_\la w_\mu^{-1}$.  The elements $w_{\la/\mu} := w_\la
w_\mu^{-1}$ provide cominuscule analogues of the 321-avoiding
permutations in type A \cite{billey.jockusch.ea:some}.  Notice that
$w_{\{\al\}}$ is a simple reflection for each $\al \in \La_X$.  In
fact, if we label each box $\al \in \La_X$ with the corresponding
simple reflection $w_{\{\al\}}$, then we recover the heap of the
longest element in $W^P$ used by Stembridge in
\cite{stembridge:fully}.  Notice also that if $\mu \subset \la \subset
\nu \subset \La_X$ are straight shapes, then $w_{\nu/\mu} =
w_{\nu/\la} w_{\la/\mu}$.

\begin{cor}
  Let $\la \subset \La_X$ be a straight shape.  Then the reduced
  expressions for $w_\la$ are exactly the expressions of the form
  $w_\la = w_{\al_{|\la|}}\cdots w_{\al_2} w_{\al_1}$, where $\al_1,
  \al_2, \dots, \al_{|\la|}$ is any ordering of the boxes of $\la$
  compatible with the partial order $\leq$ on $\La_X$.  Furthermore,
  $w_\la$ is fully commutative.
\end{cor}
\begin{proof}
  Let $\be \in \Delta$ be any simple root such that $\ell(s_\be w_\la)
  = |\la|-1$.  Then $s_\be w_\la \in W^P$, so we have $s_\be w_\la =
  w_\mu$ for some straight shape $\mu \subset \la$, such that $\la/\mu
  = \{\al\}$ is a single box.  Since $\al$ is a maximal box of $\la$
  and $w_{\{\al\}} = w_\la w_\mu^{-1} = s_\be$, it follows by
  induction on $|\la|$ that every reduced expression for $w_\la$ has
  the indicated form.  It is therefore enough to show that, if $\al,
  \al' \in \La_X$ are incomparable boxes, then the simple reflections
  $w_{\{\al\}}$ and $w_{\{\al'\}}$ commute.  To see this, let $\nu
  \subset \La_X$ be any straight shape such that $\nu\cup\{\al\}$ and
  $\nu\cup\{\al'\}$ are strictly larger straight shapes.  Then we have
  $ w_{\{\al\}} w_{\{\al'\}} = w_{\nu\cup\{\al,\al'\}/\nu\cup\{\al'\}}
  \, w_{\nu\cup\{\al'\}/\nu} = w_{\nu\cup\{\al,\al'\}/\nu} =
  w_{\nu\cup\{\al,\al'\}/\nu\cup\{\al\}} \, w_{\nu\cup\{\al\}/\nu} =
  w_{\{\al'\}} w_{\{\al\}}$, as required.
\end{proof}


\section{Increasing tableaux and jeu de taquin}\label{sec:kjdt}

\subsection{Increasing tableaux}

In this section we let $\La$ denote a partially ordered set.  The
elements of $\La$ will be called {\em boxes}.  We write $\al
\triangleleft \be$ if the box $\be$ covers $\al$, i.e.\ we have $\al <
\be$ and no box $\be' \in \La$ satisfies $\al < \be' < \be$.  We will
assume that $\La$ contains finitely many minimal boxes and each box
$\al$ has finitely many covers.  A finite lower order ideal $\la
\subset \La$ is called a {\em straight shape\/} in $\La$, and a
difference $\la/\mu := \la \ssm \mu$ of straight shapes is called a
{\em skew shape}.

\begin{defn}
  Let $\nu/\la \subset \Lambda$ be a skew shape and $S$ a set.  A {\em
    tableau\/} of shape $\nu/\la$ with values in $S$ is a map $T :
  \nu/\la \to S$.  An {\em increasing tableau\/} of shape $\nu/\la$ is
  a map $T : \nu/\la \to \Z$ such that $T(\al) < T(\be)$ for all boxes
  $\al,\be \in \nu/\la$ with $\al < \be$.
\end{defn}

We will identify the tableau $T : \nu/\la \to S$ with the filling of
the boxes of $\nu/\la$ with the values specified by $T$.  If $\La =
\La_X$ is the partially ordered set associated to a minuscule variety
$X$ as in Table~\ref{tab:shapes}, then a tableau $T : \nu/\la \to \Z$
is increasing exactly when the rows of $T$ are strictly increasing
from left to right and the columns of $T$ are strictly increasing from
top to bottom.  The shape of $T$ is denoted $\sh(T) = \nu/\la$.  If
$S' \subset S$ is a subset, then we let $T|_{S'}$ denote the
restriction of $T$ to the subset $T^{-1}(S') \subset \sh(T)$.  If
$T^{-1}(S')$ is itself a skew shape in $\La$, then $T|_{S'}$ is the
tableau obtained from $T$ by removing all boxes with values in $S \ssm
S'$.  Given straight shapes $\la \subset \mu \subset \nu$ and tableaux
$T_1 : \mu/\la \to S$ and $T_2 : \nu/\mu \to S$, we let $T_1 \cup T_2
: \nu/\la \to S$ denote their union, defined by $(T_1\cup T_2)(\al) =
T_1(\al)$ for $\al \in \mu/\la$ and $(T_1\cup T_2)(\al) = T_2(\al)$
for $\al \in \nu/\mu$.

\subsection{$K$-theoretic jeu de taquin}

We next describe the $K$-theoretic jeu de taquin algorithm of Thomas
and Yong \cite{thomas.yong:jeu}.  We will say that two boxes $\al, \be
\in \Lambda$ are {\em neighbors\/} if $\al \triangleleft \be$ or $\be
\triangleleft \al$.  Given a tableau $T$ with values in $S$ and two
elements $s,s' \in S$, define a new tableau $\swap_{s,s'}(T)$ of the
same shape by
\[
\swap_{s,s'}(T) : \al \mapsto \begin{cases}
  s' & \text{if $T(\al)=s$ and $T(\be)=s'$ for some neighbor $\be$ of $\al$;}\\
  s & \text{if $T(\al)=s'$ and $T(\be)=s$ for some neighbor $\be$ of $\al$;}\\
  T(\al) & \text{otherwise.}
\end{cases}
\]
Let $T$ be an increasing tableau of shape $\nu/\la$ with values in the
interval $[a,b] \subset \Z$ and let $C \subset \la$ be a subset of the
maximal boxes in $\la$.  Then the {\em forward slide\/} of $T$
starting from $C$ is defined by
\[
\jdt_C(T) = (\swap_{b,\bullet} \swap_{b-1,\bullet} \cdots \,
\swap_{a+1,\bullet} \swap_{a,\bullet}([C\to\bullet] \cup T)) |_\Z \,,
\]
where $[C \to \bullet]$ denotes the constant tableau of shape $C$ that
puts a dot ``$\bullet$'' in each box.  It is easy to see that
$\jdt_C(T)$ is again an increasing tableau.  Similarly, if $\wh C
\subset \Lambda \ssm \nu$ is a subset of the minimal boxes of $\Lambda
\ssm \nu$, then the {\em reverse slide\/} of $T$ starting from $\wh C$
is the increasing tableau
\[
\wh\jdt_{\wh C}(T) = (\swap_{a,\bullet} \swap_{a+1,\bullet} \cdots \,
\swap_{b-1,\bullet} \swap_{b,\bullet}(T \cup [\wh C\to\bullet])) |_\Z \,.
\]
Forward and reverse slides are inverse operations in the sense that
\[
\wh\jdt_{\wh C'}(\jdt_C(T)) \ = \ T \ = \ \jdt_{C'}(\wh\jdt_{\wh C}(T)) \,,
\]
where $\wh C' = \sh(T) \ssm \sh(\jdt_C(T))$ and $C' = \sh(T) \ssm
\sh(\wh\jdt_{\wh C}(T))$.

\begin{example}
  Let $\La = \La_{E_6/P_6}$ be the partially ordered set associated to
  the Cayley plane.  We list the sequence of intermediate tableaux
  obtained when a forward slide is applied to a tableau $T$ of shape
  $(4,4,3)/(3,1)$, starting from the box $[2,3]$ in row 2 and column
  3.  The resulting tableau $\jdt_{\{[2,3]\}}(T)$ has shape
  $(4,3,2)/(3)$.
  \[
  \tableau{10}{
   [lh]&[h]&[t]&[A]{1}\\
       &   &{\bullet}&[A]{2}&{4}&{6}\\
       &   &{3}&{4}&{5}&[r]\\
       &   &   &   &[lb]&[b]&[h]&[hr]
  }
  \!\!\mapsto\!\!
  \tableau{10}{
   [lh]&[h]&[t]&[A]{1}\\
       &   &{2}&{\bullet}&{4}&{5}\\
       &   &{3}&{4}&{5}&[r]\\
       &   &   &   &[lb]&[b]&[h]&[hr]
  }
  \!\!\mapsto\!\!
  \tableau{10}{
   [lh]&[h]&[t]&[A]{1}\\
       &   &{2}&{4}&{\bullet}&{5}\\
       &   &{3}&{\bullet}&{5}&[r]\\
       &   &   &   &[lb]&[b]&[h]&[hr]
  }
  \!\!\mapsto\!\!
  \tableau{10}{
   [lh]&[h]&[t]&[A]{1}\\
       &   &{2}&{4}&{5}&{\bullet}\\
       &   &[A]{3}&{5}&{\bullet}&[r]\\
       &   &   &   &[lb]&[b]&[h]&[hr]
  }
  \]
\end{example}
\vspace{1mm}

We will say that two increasing tableaux $S$ and $T$ are {\em jeu de
  taquin equivalent\/} if $S$ can be obtained by applying a sequence
of forward and reverse jeu de taquin slides to $T$.  The {\em jeu de
  taquin class\/} of $T$ is the set $[T]$ of all increasing tableaux
that are jeu de taquin equivalent to $T$.  If the partially ordered
set $\La$ is not understood from the context, then we will also denote
this equivalence class by $[T]_\La$.  We remark that if $\La$ is not
too large, then it is easy to generate a list of all tableaux in the
jeu de taquin class $[T]$ using a computer.  In particular, this can
be done when $\La$ is the partially ordered set associated to the
Cayley plane or the Freudenthal variety.

The following fact follows immediately from the definitions.

\begin{lemma}\label{lem:jdt_interval}
  Let $T$ and $T'$ be jeu de taquin equivalent increasing tableaux and
  let $[a,b]$ be any integer interval.  Then $T|_{[a,b]}$ and
  $T'|_{[a,b]}$ are jeu de taquin equivalent.
\end{lemma}

\subsection{Rectifications}

A {\em rectification\/} of an increasing tableau $T$ is any tableau of
straight shape that can be obtained by applying a sequence of forward
slides to $T$.  A central feature of Sch\"utzenberger's cohomological
jeu de taquin algorithm for semistandard Young tableaux is that each
semistandard tableau has a unique rectification, i.e.\ the outcome of
the jeu de taquin algorithm does not depend on the chosen inner
corners (see \cite{proctor:d-complete} and the references therein).
Thomas and Yong observed in \cite[Ex.~1.3]{thomas.yong:jeu} that this
does not hold for the $K$-theoretic jeu de taquin algorithm.  The
following example is slightly smaller.

\begin{example}\label{ex:non_urt_A}
  Let $\La = \La_{\Gr(3,6)}$.  Then the increasing tableau \
  $\tableau{10}{[lt]&[t]&[tr]\\[l]&&[atv]2\\[alh]1&[ah]3&[abr]4}$ \
  has the rectifications \
  $\tableau{10}{[alt]1&[ah]2&[ahr]4\\[avb]3&&[r]\\[lb]&[b]&[br]}$ \
  and \
  $\tableau{10}{[alt]1&[at]2&[ahr]4\\[alb]3&[abr]4&[r]\\[lb]&[b]&[br]}$\
  .\medskip
\end{example}

\begin{defn}
  An increasing tableau $U$ of straight shape is a {\em unique
    rectification target\/} if, for every increasing tableau $T$ that
  has $U$ as a rectification, $U$ is the only rectification of $T$.
\end{defn}

Unique rectification targets are essential for using the $K$-theoretic
jeu de taquin algorithm to formulate Littlewood-Richardson rules.  If
$\la \subset \La_X$ is a straight shape associated to a minuscule
variety $X$, then Thomas and Yong define the {\em row-wise
  superstandard tableau\/} of shape $\la$ to be the unique increasing
tableau $S_\la$ that fills the first row of boxes in $\la$ with the
integers $1,2,\dots,\la_1$, fills the second row with $\la_1+1, \dots,
\la_1+\la_2$, etc.  Similarly, one can define a {\em column-wise
  superstandard tableau\/} $\wh S_\la$ in which the integers increase
consecutively in each column.

\begin{example}\label{ex:superstd}
  Let $X = \OG(n,2n)$ be a maximal orthogonal Grassmannian.  Then the
  shape $\la = (5,3,2)$ has the following row-wise and column-wise
  superstandard tableaux:
  \[
  S_\la =
  \tableau{12}{{1}&{2}&{3}&{4}&{5}\\ 
  &{6}&{7}&{8}\\ 
  &&{9}&{10}}
  \hspace{5mm} \text{and} \hspace{5mm}
  \wh S_\la =
  \tableau{12}{{1}&{2}&{4}&{7}&{10}\\ 
  &{3}&{5}&{8}\\ 
  &&{6}&{9}} \ . 
  \]
\end{example}

It was proved in \cite{thomas.yong:jeu,
  clifford.thomas.ea:k-theoretic} that the row-wise superstandard
tableau $S_\la$ is a unique rectification target if $X$ is a
Grassmannian of type $A$ or a maximal orthogonal Grassmannian;
alternative proofs are given in sections \ref{sec:grass} and
\ref{sec:maxog} below.  However, the following example shows that this
fails for the Freudenthal variety.

\begin{example}\label{ex:non_urt_E7}
  Let $X = E_7/P_7$ be the Freudenthal variety and set $\la =
  (5,3,3)$.  Then the row-wise superstandard tableau $S_\la$ and the
  column-wise superstandard tableau $\wh S_\la$ both fail to be unique
  rectification targets.  In fact, consider the tableaux
  \[
  T = \tableau{10}{
    {}&{}&{}&{}&{}\\
    &&&{}&1&2&3&5\\
    &&&1&2&4&6&8\\
    &&&&&7&9&{}\\
    &&&&&&{10}&{}&{}\\
    &&&&&&{11}&{}&{}\\
    &&&&&&&&{}\\
    &&&&&&&&{}\\
    &&&&&&&&{}
  } \text{ \ \ \ \ and \ \ \ \ }
  \wh T = \tableau{10}{
    {}&{}&{}&{}&{}\\
    &&&{}&1&2&4&5\\
    &&&1&3&4&6&8\\
    &&&&&7&9&{}\\
    &&&&&&{10}&{}&{}\\
    &&&&&&{11}&{}&{}\\
    &&&&&&&&{}\\
    &&&&&&&&{}\\
    &&&&&&&&{}
  } \ .
  \]
  Then $S_\la$ is the only rectification of $\jdt_{\{[1,5]\}}(T)$, but
  $S_\la$ is not a rectification of $\jdt_{\{[2,4]\}}(T)$.  Similarly,
  $\wh S_\la$ is the only rectification of $\jdt_{\{[1,5]\}}(\wh T)$,
  but not a rectification of $\jdt_{\{[2,4]\}}(\wh T)$.
\end{example}

The next example shows that column-wise superstandard tableaux may
fail to be unique rectification targets for $\La_X$ when $X$ is a
maximal orthogonal Grassmannian.

\begin{example}\label{ex:non_urt_B}
  Let $X = \OG(6,12)$ and set $\la = (4,2)$.  Then the column-wise
  superstandard tableau $\wh S_\la$ is not a unique rectification
  target.  In fact, if $T$ is the tableau displayed below, then $\wh
  S_\la$ is a rectification of $\jdt_{\{[1,4]\}}(T)$, but $\wh S_\la$
  is not a rectification of $\jdt_{\{[2,2]\}}(T)$.
  \[
  T = \tableau{10}{
    { }&{ }&{ }&{ }&{2}\\
       &{ }&{1}&{2}&{4}\\
       &   &{3}&{5}&{ }\\
       &   &   &{6}&{ }\\
       &   &   &   &{ }
  }
  \hspace{20mm}
  \wh S_\la =
  \tableau{10}{1&2&4&6&{}\\&3&5&{}&{}\\&&{}&{}&{}\\&&&{}&{}\\&&&&{}}
  \]
\end{example}

\begin{remark}\label{rmk:all_urt}
  One can show that if $X = \Gr(m,m+k)$ with $\max(m,k) \leq 2$ then
  every increasing tableau of straight shape in $\La_X$ is a unique
  rectification target.  The same is true if $X = \OG(n,2n)$ with $n
  \leq 5$, if $X = Q^{2n}$ is an even quadric, or if $X = E_6/P_6$.
  On the other hand, examples \ref{ex:non_urt_A}, \ref{ex:non_urt_E7},
  and \ref{ex:non_urt_B} show that if $X$ is any other minuscule
  variety, then there exist tableaux of straight shapes in $\La_X$
  that are not unique rectification targets.
\end{remark}

In this paper we show that the following definition gives a uniform
construction of unique rectification targets for all minuscule
varieties.

\begin{defn}
  Given a straight shape $\la \subset \La$, define the {\em minimal
    increasing tableau\/} $M_\la$ of shape $\la$ by setting
  $M_\la(\al)$ equal to the maximal cardinality of a totally ordered
  subset of $\La$ with $\al$ as its maximal element, for each box $\al
  \in \la$.  We will say that $\La$ is a {\em unique rectification
    poset\/} if $M_\la$ is a unique rectification target for all
  straight shapes $\la \subset \La$.
\end{defn}

In other words, $M_\la$ puts the integer $1$ in the minimal boxes of
$\la$ and fills the rest of the boxes with the smallest possible
values allowed in an increasing tableau.  If $X$ is a minuscule
variety and $\La = \La_X$, then $M_\la(\al) = \height(\al)$ for each
$\al \in \la$.

\begin{example}\label{ex:mininc}
  The minimal increasing tableau of the shape $\la$ from
  Example~\ref{ex:superstd} is given by
  \[
  M_\la =
  \tableau{12}{{1}&{2}&{3}&{4}&{5}\\ 
    &{3}&{4}&{5}\\ 
    &&{5}&{6}
  } \ .\medskip
  \]
\end{example}

\begin{thm}\label{thm:mininc_urt}
  If $X$ is any minuscule variety, then $\La_X$ is a unique
  rectification poset.
\end{thm}
\begin{proof}
  We must show that if $\la \subset \La_X$ is any straight shape, then
  $M_\la$ is a unique rectification target.  This follows from
  Theorem~\ref{thm:mininc_urt_A} if $X$ is a Grassmannian of type A,
  and from Corollary~\ref{cor:mininc_urt_B} if $X$ is a maximal
  orthogonal Grassmannian.  If $X$ is a quadric hypersurface, then it
  is an easy exercise to check that every increasing tableau of
  straight shape in $\La_X$ is a unique rectification target.
  Finally, when $X$ is the Cayley plane or the Freudenthal variety, we
  have checked by computer that $M_\la$ is the only tableau of
  straight shape in the jeu de taquin class $[M_\la]$.
\end{proof}

\begin{lemma}\label{lem:urt_class}
  Let $U$ be a unique rectification target.  Then $U$ is the only
  tableau of straight shape in its jeu de taquin class $[U]$.
\end{lemma}
\begin{proof}
  Let $T \in [U]$ be any tableau of straight shape, and choose a
  sequence $U = T_0, T_1, \dots, T_m = T$ such that each tableau $T_i$
  can be obtained by applying a slide to $T_{i-1}$.  Since $T_{i-1}$
  and $T_i$ share at least one rectification, it follows by induction
  that $U$ is the only rectification of $T_i$ for each $i$.  It
  follows that $T = U$.
\end{proof}

\subsection{$K$-infusion}

Jeu de taquin slides are special cases of the more general {\em
  $K$-infusion algorithm\/} from \cite{thomas.yong:jeu}, which can be
defined as follows.  For each integer $a \in \Z$ we introduce a {\em
  barred\/} copy $\ov{a}$, and we set $\ov{\ov a} = a$.  Let $\ov\Z =
\{ \ov{a} \mid a \in \Z \}$ be the set of all barred integers.  If $Y$
is a tableau with values in $\Z \cup \ov\Z$, then we let $\ov Y$
denote the tableau obtained by replacing $Y(\al)$ with $\ov{Y(\al)}$
for each $\al \in \sh(Y)$.  Define an automorphism of the set of all
tableaux with values in $\Z \cup \ov\Z$ by
\[
\Psi = \prod_{a=-\infty}^\infty \prod_{b=\infty}^{-\infty}
\swap_{\ov a,b} \,.
\]
This is well defined because only finitely many factors of $\Psi$ will
result in changes to any given tableau.  Furthermore, since
$\swap_{\ov a,b}$ commutes with $\swap_{\ov{a'},b'}$ whenever $a \neq
a'$ and $b \neq b'$, the factors of $\Psi$ can be rearranged in any
order such that $\swap_{\ov a,b}$ is applied after
$\swap_{\ov{a'},b'}$ whenever $a \leq a'$ and $b \geq b'$.  This
implies that $\Psi = \prod_{b=\infty}^{-\infty}
\prod_{a=-\infty}^\infty \swap_{\ov a,b}$, so we have $\ov{\Psi(Y)} =
\Psi^{-1}(\ov{Y})$, hence the map $Y \mapsto \ov{\Psi(Y)}$ is an
involution on the set of all tableaux $Y$ with values in $\Z\cup
\ov\Z$.

\begin{example}
  The following tableaux are mapped to each other by the involution.
  \[
  Y = \text{\small$\tableau{12}{[lh]&[h]&{}&[aTV]\ov{2}\\
  &&[aLT]\ov{1}&[aB]\ov{3}&[aHR]\ov{4}&[aTV]1\\ 
  &&[aVB]\ov{3}&[aLH]1&[a]2&[aR]3\\ &&&&[aLB]3&[aB]4&[aHR]5&[a]}$}
  \hspace{7mm} ; \hspace{7mm}
  \ov{\Psi(Y)} = \text{\small$\tableau{12}{[lh]&[h]&{}&[aTV]\ov{1}\\
  &&[aLT]\ov{1}&[a]\ov{2}&[aT]\ov{3}&[aHR]\ov{4}\\ 
  &&[aLB]\ov{3}&[aB]\ov{4}&[a]\ov{5}&[aTV]2\\ &&&&[aLH]1&[aB]3&[aHR]4&[a]}$}
  \]
\end{example}

Given increasing tableaux $S$ and $T$ such that $\sh(S) =
\mu/\la$ and $\sh(T) = \nu/\mu$ for straight shapes $\la \subset \mu
\subset \nu$, we define
\begin{equation}\label{eqn:kinfusion}
  \jdt_S(T) := \Psi(\ov S \cup T)|_\Z \text{ \ \ \ \ and \ \ \ \ }
  \wh\jdt_T(S) := \Psi^{-1}(S \cup \ov{T})|_{\Z} \,.
\end{equation}
The following result is proved in \cite[Thm.~3.1]{thomas.yong:jeu}.
We sketch the short proof for convenience.

\begin{prop}[Thomas and Yong]\label{prop:kinfusion}
  Let $\la \subset \nu$ be straight shapes in $\La$.  Then the map
  $(S,T) \mapsto (\jdt_S(T), \wh\jdt_T(S))$ is an involution on the
  set of all pairs $(S,T)$ of increasing tableaux for which
  $\sh(S)=\mu/\la$ and $\sh(T)=\nu/\mu$ for some straight shape $\mu$.
  Furthermore, the tableau $\jdt_S(T)$ can be obtained by applying a
  sequence of forward slides to $T$, and $\wh\jdt_T(S)$ is obtained by
  applying a sequence of reverse slides to $S$.
\end{prop}
\begin{proof}
  For each $m \in \Z$ we define the tableau
  \[
  Y_m = \left( \prod_{a=m+1}^\infty \prod_{b=\infty}^{-\infty}
    \swap_{\ov a,b} \right) (\ov S \cup T) \,.
  \]
  Then we have $Y_{m-1}|_\Z = \jdt_{C_m}(Y_m|_\Z)$, where $C_m =
  Y_m^{-1}(\ov m) = S^{-1}(m)$ is the set of boxes mapped to $m$ by
  $S$.  It follows by descending induction on $m$ that
  $\ov{Y_m}|_{[-\infty,m]}$, $Y_m|_\Z$, and $\ov{Y_m}|_{[m+1,\infty]}$
  are increasing tableaux, with shapes given by
  $\sh(\ov{Y_m}|_{[-\infty,m]}) = \mu'_m/\la$, $\sh(Y_m|_\Z) =
  \mu_m/\mu'_m$, and $\sh(\ov{Y_m}|_{[m+1,\infty]}) = \nu/\mu_m$,
  where $\mu'_m$ and $\mu_m$ are straight shapes satisfying $\la
  \subset \mu'_m \subset \mu_m \subset \nu$.  In addition, $Y_m|_\Z$
  can be obtained from $T$ by a sequence of forward slides.  If we
  choose $m$ smaller than all the integers in $S$, then $Y_m =
  \Psi(\ov{S} \cup T)$, so that $\jdt_S(T) = Y_m|_\Z$ and
  $\wh\jdt_T(S) = \ov{Y_m}|_\Z$.  We deduce that $\jdt_S(T)$ and
  $\wh\jdt_T(S)$ are increasing tableaux of shapes $\sh(\jdt_S(T)) =
  \mu_m/\la$ and $\sh(\wh\jdt_T(S)) = \nu/\mu_m$, and $\jdt_S(T)$ can
  be obtained from $T$ by a sequence of forward slides.  The identity
  $\ov{\jdt_S(T)} \cup \wh\jdt_T(S) = \ov{\Psi(\ov{S} \cup T)}$ and
  the fact that $Y \mapsto \ov{\Psi(Y)}$ is an involution implies that
  $(S,T) \mapsto (\jdt_S(T),\wh\jdt_T(S))$ is an involution.  Finally,
  set $(T',S') = (\jdt_S(T),\wh\jdt_T(S))$ and notice that $S =
  \jdt_{T'}(S')$.  This shows that $S$ can be obtained by applying a
  sequence of forward slides to $S'$, or equivalently, $S' =
  \wh\jdt_T(S)$ is obtained by applying a sequence of reverse slides
  to $S$.
\end{proof}

\subsection{Combinatorial $K$-theory rings}\label{sec:comb_ktheory}

Assume that $\La$ is a unique rectification poset.  For each straight
shape $\la$ in $\La$, let $G_\la$ be a symbol, and let $\Gamma(\La) =
\bigoplus_\la \Z\,G_\la$ be the free abelian group generated by these
symbols.  We will say that a set $\tau$ of increasing tableaux is {\em
  locally finite\/} if all tableaux of $\tau$ have values in some
finite interval $[-c,c]$.  Assume that $\tau$ is locally finite and
closed under (forward and reverse) jeu de taquin slides.  For any skew
shape $\nu/\la$ we let $\tau(\nu/\la) = \# \{ T \in \tau \mid
\sh(T)=\nu/\la\}$ be the number of tableaux of shape $\nu/\la$ in
$\tau$.  Define a linear operator $F_\tau : \Gamma(\La) \to
\Gamma(\La)$ by $F_\tau(G_\la) = \sum_\nu \tau(\nu/\la)\, G_\nu$.  The
assumption that $\La$ has finitely many minimal boxes and each box of
$\La$ has finitely many covers implies that $F_\tau(G_\la)$ is a
finite linear combination.

\begin{lemma}\label{lem:commute}
  Let $\sigma$ and $\tau$ be locally finite sets of skew increasing
  tableaux, each closed under jeu de taquin slides.  Then the
  operators $F_\sigma$ and $F_\tau$ commute.
\end{lemma}
\begin{proof}
  The coefficient of $G_\nu$ in $F_\sigma(F_\tau(G_\la))$ is equal to
  $\sum_\mu \sigma(\nu/\mu) \tau(\mu/\la)$, and the coefficient of
  $G_\nu$ in $F_\tau(F_\sigma(G_\la))$ is equal to $\sum_\mu
  \tau(\nu/\mu) \sigma(\mu/\la)$.  It follows from
  Proposition~\ref{prop:kinfusion} that these sums are equal to each
  other.
\end{proof}

For each straight shape $\la \subset \La$ we set $F_\la =
F_{[M_\la]}$.  It follows from Lemma~\ref{lem:urt_class} that the
coefficient of $G_\nu$ in $F_\la(G_\mu)$ is the number of increasing
tableaux of shape $\nu/\mu$ that rectify to the minimal increasing
tableau $M_\la$.

\begin{prop}\label{prop:assoc}
  Assume that $\La$ is a unique rectification poset.  Then the
  bilinear map $\Gamma(\La) \times \Gamma(\La) \to \Gamma(\La)$
  defined by $G_\la \cdot G_\mu = F_\la(G_\mu)$ gives $\Gamma(\La)$
  the structure of a commutative and associative ring with unit $1 =
  G_\emptyset$.  Furthermore, for any locally finite set $\tau$ of
  increasing tableaux that is closed under jeu de taquin slides and
  any $f \in \Gamma(\La)$, we have $F_\tau(f) = F_\tau(1)\cdot f$.
\end{prop}
\begin{proof}
  It follows from Lemma~\ref{lem:urt_class} that $F_\la(1) = G_\la$,
  and Lemma~\ref{lem:commute} implies that $F_\tau(G_\la) =
  F_\tau(F_\la(1)) = F_\la(F_\tau(1)) = G_\la \cdot F_\tau(1)$.  It
  follows that $F_\tau(f) = f \cdot F_\tau(1)$ for all $f \in
  \Gamma(\La)$ by linearity.  Commutativity follows from this because
  $G_\la \cdot G_\mu = F_\la(G_\mu) = G_\mu \cdot F_\la(1) = G_\mu
  \cdot G_\la$, and associativity follows because $G_\la \cdot (G_\mu
  \cdot G_\nu) = F_\la(F_\mu(F_\nu(1))) = F_\nu(F_\la(F_\mu(1))) =
  G_\nu \cdot (G_\la \cdot G_\mu) = (G_\la \cdot G_\mu) \cdot G_\nu$.
\end{proof}

\begin{example}
  Let $X = E_6/P_6$ be the Cayley plane.  Then we have $G_2 \cdot G_2
  = G_{(4)} + G_{(3,1)} + G_{(4,1)}$ in $\Gamma(\La_X)$, due to the
  following tableaux that all rectify to $M_{(2)}$.
  \[
  \tableau{10}{[lh]&[h]&[A]1&2\\ &&[l]&[t]&{}&[tr]\\ 
    &&[lb]&[b]&&[r]\\ &&&&[lb]&[b]&[h]&[hr]}
  \hspace{5mm}
  \tableau{10}{[lh]&[h]&[A]1&[tr]\\ &&[A]2&&[t]&[tr]\\ 
    &&[lb]&[b]&&[r]\\ &&&&[lb]&[b]&[h]&[hr]}
  \hspace{5mm}
  \tableau{10}{[lh]&[h]&[A]1&2\\ &&2&&[t]&[tr]\\
    &&[lb]&[b]&&[r]\\ &&&&[lb]&[b]&[h]&[hr]}\medskip
  \]
\end{example}

The {\em structure constants\/} of the ring $\Gamma(\La)$ are the
integers $c^\nu_{\la,\mu} = c^\nu_{\la,\mu}(\La) := [M_\la](\nu/\mu)$.
Equivalently, the identity $G_\la \cdot G_\mu = \sum_\nu
c^\nu_{\la,\mu}\, G_\nu$ holds in $\Gamma(X)$.

The next two results have appeared earlier when $\La$ is the partially
ordered set associated to a Grassmann variety of type A or a maximal
orthogonal Grassmannian, the constants $c^\nu_{\la,\mu}$ are the
$K$-theoretic Schubert structure constants for this space, and $U$ is
a superstandard tableau.  Corollary~\ref{cor:lrrule} was proved in
\cite{thomas.yong:jeu} for Grassmannians of type A and in
\cite{buch.ravikumar:pieri, clifford.thomas.ea:k-theoretic} for
maximal orthogonal Grassmannians.  Corollary~\ref{cor:lrUT0} was
obtained in \cite{thomas.yong:direct}.

\begin{cor}\label{cor:lrrule}
  Let $\la$, $\mu$, and $\nu$ be straight shapes, and let $U$ be any
  unique rectification target of shape $\mu$.  Then $c^\nu_{\la,\mu}$
  is equal to the number of increasing tableaux of shape $\nu/\la$
  that rectify to $U$.
\end{cor}
\begin{proof}
  This is true because $G_\la \cdot G_\mu = G_\la \cdot F_{[U]}(1) =
  F_{[U]}(G_\la)$.
\end{proof}

\begin{cor}\label{cor:lrUT0}
  Let $\la$, $\mu$, and $\nu$ be straight shapes, let $T_0$ be any
  increasing tableau of shape $\mu$, and let $U$ be any unique
  rectification target of shape $\la$.  Then $c^\nu_{\la,\mu}$ is the
  number of increasing tableaux $T$ of shape $\nu/\la$ for which
  $\jdt_U(T) = T_0$.
\end{cor}
\begin{proof}
  It follows from Proposition~\ref{prop:kinfusion} that the map
  $T \mapsto \wh\jdt_T(U)$ is a bijection from the set of tableaux
  $T$ of shape $\nu/\la$ for which $\jdt_U(T)=T_0$ to the set of
  tableaux of shape $\nu/\mu$ that rectify to $U$.
\end{proof}

\begin{example}\label{ex:e8_fails}
  Let $X = E_7/P_7$ be the Freudenthal variety and consider the shapes
  $\la = (5,1)$, $\mu = (5,3,3)$, and $\nu = (5,5,5,2,1,1)$ in
  $\La_X$.  Then $c^\nu_{\la,\mu}(\La_X) = 11$.  There are 12
  increasing tableaux of shape $\nu/\la$ that have $S_\mu$ as a
  rectification, and 10 of these tableaux have $S_\mu$ as the only
  rectification.  Similarly, there are 12 increasing tableaux of shape
  $\nu/\la$ that have $\wh S_\mu$ as a rectification, and 10 of these
  tableaux have $\wh S_\mu$ as the only rectification.
\end{example}

\begin{remark}
  Proposition~\ref{prop:assoc} would be true with the weaker
  hypothesis that for each straight shape $\la \subset \La$ there
  exists {\em some\/} unique rectification target $U_\la$ of shape
  $\la$, and we set $F_\la = F_{[U_\la]}$.  However, all partially
  ordered sets $\La$ with this property that we know about are unique
  rectification posets.
\end{remark}

\begin{remark}
  Let $\La' \subset \La$ be any lower order ideal and $T$ an
  increasing tableau of straight shape $\sh(T) \subset \La'$.  If $T$
  is a unique rectification target for $\La$ then $T$ is also a unique
  rectification target for $\La'$.  In particular, if $\La$ is a
  unique rectification poset, then so is $\La'$.
\end{remark}

\begin{remark}
  Proctor has defined a notion of {\em $d$-complete posets\/} that
  generalize the partially ordered sets associated to minuscule
  varieties.  He proves in \cite{proctor:d-complete} that every
  $d$-complete poset has the {\em jeu de taquin property}, i.e.\ the
  rectification of any standard skew tableau with Sch\"utzenberger's
  cohomological jeu de taquin algorithm is independent of choices.  It
  would be interesting to know if all $d$-complete posets are also
  unique rectification posets.  We have checked that this is the case
  for a few $d$-complete posets that are not associated to minuscule
  varieties.
\end{remark}


\section{$K$-theory of minuscule varieties}\label{sec:ktheory}

\subsection{The Grothendieck ring}

In this section we let $X = G/P$ be a minuscule variety and $\La_X$
the associated partially ordered set (see
Remark~\ref{rmk:coroot_poset}).  Theorem~\ref{thm:mininc_urt} and
Proposition~\ref{prop:assoc} imply that the combinatorial $K$-theory
ring $\Gamma(\La_X)$ is a well defined associative ring.  Let $K(X)$
be the Grothendieck ring of algebraic vector bundles on $X$.  A brief
description of this ring can be found in \cite[\S
2]{buch.ravikumar:pieri}.  In this section we first show that the ring
$K(X)$ is isomorphic to $\Gamma(\La_X)$.  We then use the geometry of
$X$ to prove results about about increasing tableaux for $\La_X$.

For each straight shape $\la \subset \La_X$ we let $\cO_\la :=
[\cO_{X_\la}] \in K(X)$ denote the Grothendieck class of the Schubert
variety $X_\la$.  The next result shows that the $K$-theoretic
Schubert structure constants of $X$ are the integers
$(-1)^{|\nu|-|\la|-|\mu|}\,c^\nu_{\la,\mu}(\La_X)$, i.e.\ the
identity $\cO_\la \cdot \cO_\mu = \sum_\nu (-1)^{|\nu|-|\la|-|\mu|} \,
c^\nu_{\la,\mu}(\La_X) \, \cO_\nu$ holds in $K(X)$.

\begin{thm}\label{thm:iso_groth}
  Let $X$ be a minuscule variety.  Then the $\Z$-linear isomorphism
  $\varphi : \Gamma(\La_X) \to K(X)$ defined by $G_\la \mapsto
  (-1)^{|\la|}\, \cO_{X_\la}$ is an isomorphism of rings.
\end{thm}
\begin{proof}
  It is enough to show that there exist straight shapes
  $\mu_1,\dots,\mu_k$ and $\nu_1,\dots,\nu_\ell$ in $\La_X$ such that
  the following conditions hold.\smallskip

  \noin (a) \ $\varphi(G_{\mu_j} \cdot G_\la) = \varphi(G_{\mu_j})
  \cdot \varphi(G_\la)$ for every straight shape $\la \subset \La_X$
  and $1 \leq j \leq k$.\smallskip

  \noin (b) \ $\varphi(G_{\nu_i} \cdot G_{\nu_j}) = \varphi(G_{\nu_i})
  \cdot \varphi(G_{\nu_j})$ for all $1 \leq i \leq j \leq
  \ell$.\smallskip

  \noin (c) \ The cohomology ring $H^*(X;\Q)$ is spanned by the
  classes $1, [X_{\nu_1}],\dots,[X_{\nu_\ell}]$ as a module over the
  subring of $H^*(X;\Q)$ generated by $[X_{\mu_1}], \dots,
  [X_{\mu_k}]$.\smallskip
  
  To see that these conditions are sufficient, notice first that since
  the cohomology ring $H^*(X;\Q)$ is the associated graded ring of
  $K(X)_\Q := K(X) \otimes_\Z \Q$ modulo Grothendieck's gamma
  filtration (see examples 15.1.5 and 15.2.16 in
  \cite{fulton:intersection}), condition (c) implies that $K(X)_\Q$ is
  spanned by $\cO_{\nu_0},\dots,\cO_{\nu_\ell}$ as a module over the
  subring $R = \Q[\cO_{\mu_1},\dots,\cO_{\mu_k}] \subset K(X)_\Q$,
  where we set $\nu_0 = \emptyset$ so that $\cO_{\nu_0}=1$.  Set $S =
  \Q[G_{\mu_1},\dots,G_{\mu_k}] \subset \Gamma(\La_X)_\Q$.  Then (a)
  implies that $\varphi : S \to R$ is an isomorphism of rings, and
  that $\varphi(s \cdot f) = \varphi(s) \cdot \varphi(f)$ for all $s
  \in S$ and $f \in \Gamma(\La_X)_\Q$.  It follows that
  $\Gamma(\La_X)_\Q$ is spanned by $G_{\nu_0},\dots,G_{\nu_\ell}$ as
  an $S$-module.  Finally, given $f, g \in \Gamma(\La_X)_\Q$ we may
  write $f = \sum a_i G_{\nu_i}$ and $g = \sum b_j G_{\nu_j}$ with
  $a_i, b_j \in S$.  Using (b) we obtain
  \[
  \varphi(f g) 
  = \sum_{i,j} \varphi(a_i) \varphi(b_j) \varphi(G_{\nu_i} G_{\nu_j}) 
  = \sum_{i,j} \varphi(a_i) \varphi(b_j) \varphi(G_{\nu_i}) \varphi(G_{\nu_j})
  = \varphi(f) \varphi(g) \,,
  \]
  as required.
  
  Assume first that $X = \Gr(m,n)$ is a Grassmannian of type A.  In
  this case we can use the shapes $\mu_i = (i)$ for $1 \leq i \leq
  n-m$.  Since the special Schubert classes $[X_{\mu_i}]$ generate
  $H^*(X;\Q)$, no shapes $\nu_j$ are required for property (c).
  Property (a) follows by comparing Lenart's Pieri rule
  \cite[Thm.~3.4]{lenart:combinatorial} for products of the form
  $\cO_{(i)} \cdot \cO_\la$ in $K(X)$ to the corresponding Pieri rule
  for $\Gamma(\La_X)$ given in Corollary~\ref{cor:pieri_A}.

  Assume next that $X = \OG(n,2n)$ is a maximal orthogonal
  Grassmannian.  We use $\mu_i = (i)$ for $1 \leq i \leq n-1$; no
  shapes $\nu_j$ are required.  Property (a) follows by comparing the
  $K$-theoretic Pieri rule for $K(X)$ proved in
  \cite[Cor.~4.8]{buch.ravikumar:pieri} to
  Corollary~\ref{cor:pieri_B}.

  Assume that $X = E_6/P_6$ is the Cayley Plane.  Here we use $\mu_1 =
  (1)$, $\nu_1 = (4)$, and $\nu_2 = (4,4)$.  Property (a) can be
  verified using Lenart and Postnikov's $K$-theoretic Chevalley
  formula \cite{lenart.postnikov:affine}.  The Betti numbers $d_i :=
  \dim H^{2i}(X;\Q)$ of $X$ are given by $d_i=1$ for $0 \leq i \leq
  3$, $d_i=2$ for $4 \leq i \leq 7$, $d_8=3$, and $d_{16-i}=d_i$ for
  $0 \leq i \leq 16$.  Using the Chevalley formula, one can check that
  $[X_{(4)}] \notin [X_{(1)}] \cdot H^6(X;\Q)$ and that $[X_{(4,4)}]
  \notin [X_{(1)}] \cdot H^{14}(X;\Q)$.  Property (c) therefore
  follows from the Hard Lefschetz theorem.  We have computed the
  products $\cO_{\nu_i} \cdot \cO_{\nu_j}$ in the equivariant
  $K$-theory ring $K_T(X)$, where $T \subset G$ is a maximal torus.
  This was done by utilizing that each such product is uniquely
  determined by its restrictions to the $T$-fixed points of $X$
  \cite{kostant.kumar:t-equivariant}.  Formulas for the restrictions
  of Schubert classes to $T$-fixed points can be found in
  \cite{graham:equivariant, willems:k-theorie, knutson:schubert}.  The
  products $\cO_{\nu_i} \cdot \cO_{\nu_j}$ in the ordinary $K$-theory
  ring are given by
  \[
  \begin{split}
    \cO_{(4)} \cdot \cO_{(4)} &\ = \ \cO_{(4,4)} + \cO_{(4,3,1)} +
    \cO_{(4,2,2)} - \cO_{(4,4,1)} - \cO_{(4,3,2)}
    \ , \\
    \cO_{(4)} \cdot \cO_{(4,4)} &\ = \ \cO_{(4,4,4)} \ , \\
    \cO_{(4,4)} \cdot \cO_{(4,4)} &\ = \ \cO_{(4,4,4,4)} \,.
  \end{split}
  \]
  Property (b) follows from this by observing that the following seven
  tableaux are the only increasing tableaux that contribute to these
  products.\medskip

  \noin \ \ 
  $\tableau{10}{{}&{}&{}&{}\\ &&{1}&{2}&{3}&{4}\\ 
    &&{}&{}&{}&{}\\ &&&&{}&{}&{}&{}}$ \ \ 
  $\tableau{10}{{}&{}&{}&{}\\ &&{1}&{2}&{4}&{}\\ 
    &&{3}&{}&{}&{}\\ &&&&{}&{}&{}&{}}$ \ \ 
  $\tableau{10}{{}&{}&{}&{}\\ &&{1}&{2}&{}&{}\\ 
    &&{3}&{4}&{}&{}\\ &&&&{}&{}&{}&{}}$ \ \ 
  $\tableau{10}{{}&{}&{}&{}\\ &&{1}&{2}&{3}&{4}\\ 
    &&{3}&{}&{}&{}\\ &&&&{}&{}&{}&{}}$
  \medskip
  
  \noin \ \ 
  $\tableau{10}{{}&{}&{}&{}\\ &&{1}&{2}&{4}&{}\\ 
    &&{3}&{4}&{}&{}\\ &&&&{}&{}&{}&{}}$ \ \ 
  $\tableau{10}{{}&{}&{}&{}\\ &&{}&{}&{}&{}\\ 
    &&{1}&{2}&{3}&{4}\\ &&&&{}&{}&{}&{}}$ \ \ 
  $\tableau{10}{{}&{}&{}&{}\\ &&{}&{}&{}&{}\\ 
    &&{1}&{2}&{3}&{4}\\ &&&&{4}&{5}&{6}&{7}}$
  \medskip

  If $X = E_7/P_7$ is the Freudenthal variety, then the proof is
  analogous to the argument given for the Cayley plane.  We use $\mu_1
  = (1)$, $\nu_1=(5)$, and $\nu_2 = (5,4)$.  The Betti numbers $d_i :=
  \dim H^{2i}(X;\Q)$ are given by $d_i = 1$ for $0 \leq i \leq 4$,
  $d_i = 2$ for $5 \leq i \leq 8$, $d_i = 3$ for $9 \leq i \leq 18$,
  and $d_{27-i} = d_i$ for $0 \leq i \leq 27$.  The products
  $\cO_{\nu_i} \cdot \cO_{\nu_j}$ can be computed with equivariant
  methods and are given by
  \[
  \begin{split}
    \cO_{(5)} \cdot \cO_{(5)} &\ = \ 2\,\cO_{(5,4,1)} + 2\,\cO_{(5,3,2)} -
    3\,\cO_{(5,4,2)} - \cO_{(5,3,3)}
    + \cO_{(5,4,3)} \ , \\
    \cO_{(5)}\cdot \cO_{(5,4)} &\ = \ 2\,\cO_{(5,5,4)} +
    2\,\cO_{(5,5,3,1)} + \cO_{(5,4,4,1)}
    - 4\,\cO_{(5,5,4,1)} \ , \\
    \cO_{(5,4)} \cdot \cO_{(5,4)} &\ = \ 2\,\cO_{(5,5,5,2,1)} +
    2\,\cO_{(5,5,4,2,1,1)} - 3\,\cO_{(5,5,5,2,1,1)} \,.
  \end{split}
  \]
  We leave it as an exercise to identify the corresponding 25
  increasing tableaux.
  
  Finally, assume that $X = Q^{2n} \subset \bP^{2n+1}$ is an even
  quadric hypersurface.  Fix the orthogonal form on $\C^{2n+2}$
  defined by $(e_i,e_j) = \delta_{i+j,2n+3}$.  Then $X = \{ L \in
  \bP^{2n+1} \mid (L,L) = 0 \}$.  For $0 \leq p \leq 2n$ we let $X_p
  \subset X$ be the subvariety defined by
  \[
  X_p = \begin{cases}
    \bP^{2n+1-p} \cap X & \text{if $0 \leq p < n$,}\\
    \bP^{2n-p} & \text{if $n \leq p \leq 2n$,}
  \end{cases}
  \]
  where $\bP^m \subset \bP^{2n-1}$ denotes the linear subspace spanned
  by $e_1,\dots,e_{m+1}$.  Let $X'_n \subset \bP^{2n+1}$ be the linear
  subspace spanned by $e_1,\dots,e_n,e_{n+2}$.  Then $X_p \subset X$
  is the unique Schubert variety of codimension $p$ when $p \neq n$,
  while $X_n$ and $X'_n$ are the Schubert varieties of codimension
  $n$.  We claim that multiplication with the divisor class
  $[\cO_{X_1}]$ in $K(X)$ is determined by $[\cO_{X_1}] \cdot
  [\cO_{X'_n}] = [\cO_{X_{n+1}}]$ and
  \begin{equation}\label{eqn:quadchev}
    [\cO_{X_1}] \cdot [\cO_{X_p}] = \begin{cases}
      [\cO_{X_{p+1}}] & \text{if $p \notin \{n-1,2n\}$;} \\
      [\cO_{X_n}] + [\cO_{X'_n}] - [\cO_{X_{n+1}}] &\text{if $p=n-1$;}\\
      0 & \text{if $p=2n$.}
    \end{cases}
  \end{equation}
  To see this, consider the linear map $i_* : K(X) \to K(\bP^{2n+1})$
  defined by the inclusion $i : X \subset \bP^{2n+1}$.  If we let $h =
  [\cO_H] \in K(\bP^{2n+1})$ be the class of a hyperplane, then we
  have $i_*[\cO_{X_p}] = h^p(2h-h^2)$ for $0 \leq p < n$ and
  $i_*[\cO_{X_p}] = h^{p+1}$ for $n \leq p \leq 2n$.  Since the
  projection formula implies that $i_*([\cO_{X_1}] \cdot [\cO_{X_p}])
  = h \cdot i_*[\cO_{X_p}]$, it follows that we obtain a valid
  identity in $K(\bP^{2n+1})$ after applying $i_*$ to both sides of
  (\ref{eqn:quadchev}).  The claim now follows because $\Ker(i_*) =
  \Z\,([\cO_{X_n}] - [\cO_{X'_n}])$ and both sides of
  (\ref{eqn:quadchev}) are invariant under the involution of $K(X)$
  that interchanges $[\cO_{X_n}]$ and $[\cO_{X'_n}]$.
  
  We also claim that
  \begin{equation}\label{eqn:quadop}
    [\cO_{X_n}] \cdot [\cO_{X_n}] = \begin{cases}
      [\cO_{X_{2n}}] = [\cO_{\text{point}}] & \text{if $n$ is even;}\\
      0 & \text{if $n$ is odd.}
    \end{cases}
  \end{equation}
  This follows because a variety $\bP(E) \subset \bP^{2n+1}$ is a
  translate of $X_n$ under the action of $G = \SO(2n+2)$ if and only
  if $E \subset \C^{2n+2}$ is a linear subspace of dimension $n+1$
  such that $\dim(E \cap \C^{n+1}) \equiv n+1$ (mod 2).  Since the
  opposite Schubert variety $X_n^\op$ is a $G$-translate of $X_n$, we
  obtain
  \[
  X_n^\op = \begin{cases}
    \bP(\Span\{e_{n+2},\dots,e_{2n+2}\}) & \text{if $n$ is odd,} \\
    \bP(\Span\{e_{n+1},e_{n+3},\dots,e_{2n+2}\}) & \text{if $n$ is
      even,}
  \end{cases}
  \]
  so $X_n \cap X_n^\op = \{ \C\, e_{n+1} \}$ is a single point when
  $n$ is even, while $X_n \cap X_n^\op = \emptyset$ when $n$ is odd.
  We leave it to the reader to check that properties (a), (b), and (c)
  for the shapes $\mu_1 = (1)$ and $\nu_1 = (n)$ follow from
  (\ref{eqn:quadchev}) and (\ref{eqn:quadop}).
\end{proof}

\subsection{Duality}

The sheaf Euler characteristic map $\euler{X} : K(X) \to \Z$ is the
$\Z$-linear map defined by $\euler{X}(\cO_\la) = 1$ for all straight
shapes $\la$.  Since non-empty Richardson varieties are rational
\cite{richardson:intersections} and have rational singularities
\cite{brion:positivity}, we deduce from Corollary~\ref{cor:poincare}
that we have
\begin{equation}\label{eqn:richardson}
  \euler{X}(\cO_\la \cdot \cO_\mu) = 
  \begin{cases}
    1 & \text{if $\la \subset \mu^\vee$;}\\
    0 & \text{otherwise,}
  \end{cases}
\end{equation}
for all pairs of straight shapes $\la$ and $\mu$ (see \cite[\S
2]{buch.ravikumar:pieri} for more details).

A subset $\theta \subset \La_X$ is called a {\em rook strip\/} if no
two boxes of $\theta$ are comparable by the partial order $\leq$.
Equivalently, $\theta$ is a skew shape such that no two boxes of
$\theta$ belong to the same row or column when $\theta$ is identified
with a subset of $\N^2$.  Define the {\em dual Grothendieck class\/}
of $X_\la$ by
\begin{equation}\label{eqn:dualclass}
  \cO^*_\la \ = \ \sum_{\nu/\la^\vee \text{ rook strip}}
  (-1)^{|\nu/\la^\vee|}\, \cO_\nu \, \in K(X)
\end{equation}
where the sum is over all straight shapes $\nu$ containing $\la^\vee$
such that $\nu/\la^\vee$ is a rook strip.  The identity
(\ref{eqn:richardson}) implies that
\begin{equation}\label{eqn:dualpair}
  \euler{X}(\cO_\la \cdot \cO^*_\mu) = \delta_{\la,\mu} \,.
\end{equation}

We remark that the identities (\ref{eqn:richardson}),
(\ref{eqn:dualclass}), (\ref{eqn:dualpair}) are valid on any
cominuscule homogeneous space.  However, the following result fails
for the the Lagrangian Grassmannian $\LG(n,2n)$ by
\cite[Example~5.7]{buch.ravikumar:pieri}.  Set $c^\nu_{\la,\mu} =
c^\nu_{\la,\mu}(\La_X)$.

\begin{thm}\label{thm:symmetry}
  Let $X$ be a minuscule variety and let $\la$, $\mu$, and $\nu$ be
  straight shapes.  Then we have $\cO_\la^* = (1 - \cO_{(1)}) \cdot
  \cO_{\la^\vee}$ and $c^\nu_{\la,\mu} = c^{\mu^\vee}_{\la,\nu^\vee}$.
\end{thm}
\begin{proof}
  The identity $\cO_\la^* = (1 - \cO_{(1)}) \cdot \cO_{\la^\vee}$
  follows from Theorem~\ref{thm:iso_groth}.  We obtain
  \[
  (-1)^{|\nu|-|\la|-|\mu|}\, c^\nu_{\la,\mu} =
  \euler{X}(\cO_\la \cdot \cO_\mu \cdot \cO_\nu^*) =
  \euler{X}((1-\cO_{(1)}) \cdot \cO_\la \cdot \cO_\mu \cdot
  \cO_{\nu^\vee}) \,,
  \]
  which implies that $c^\nu_{\la,\mu} = c^{\mu^\vee}_{\la,\nu^\vee}$.
\end{proof}

Theorem~\ref{thm:symmetry} implies that the constants $c(\la,\mu,\nu)
= c^{\nu^\vee}_{\la,\mu}$ are invariant under arbitrary permutations
of $\la$, $\mu$, and $\nu$.  This was known for Grassmannians of type
A \cite[Cor.~1]{buch:combinatorial} and for maximal orthogonal
Grassmannians \cite[Cor.~4.6]{buch.ravikumar:pieri}.

\subsection{Unique rectification targets}

We will say that a subset of $\La_X$ is an {\em anti-straight shape\/}
if it is an upper order ideal.  Any anti-straight shape has the form
$w_X.\la = \La_X/\la^\vee$ for some straight shape $\la$ (see
\S\ref{sec:comin}).  An {\em anti-rectification\/} of an increasing
tableau $T$ is any tableau of anti-straight shape that can be obtained
by applying a sequence of reverse slides to $T$.

We require the following involution on the set of all increasing
tableaux.  If $T$ is an increasing tableau of shape $\nu/\la$, then we
let $w_X.T$ be the increasing tableau of shape $w_X.(\nu/\la) =
\la^\vee/\nu^\vee$ defined by $(w_X.T)(\al) = - T(w_X.\al)$.  This
involution commutes with jeu de taquin slides in the sense that
$w_X.\jdt_C(T) = \wh\jdt_{w_X.C}(w_X.T)$.

\begin{thm}\label{thm:urt_conds}
  Let $T$ be an increasing tableau of straight shape.  The following
  are equivalent.

  \noin{\rm(a)} $T$ is a unique rectification target.

  \noin{\rm(b)} $T$ has exactly one anti-rectification.

  \noin{\rm(c)} The jeu de taquin class $[T]$ contains exactly one
  tableau of straight shape.

  \noin{\rm(d)} The jeu de taquin class $[T]$ contains exactly one
  tableau of anti-straight shape.

  \noin{\rm(e)} All tableaux of $[T]$ can be obtained by applying a
  sequence of reverse slides to $T$.
\end{thm}
\begin{proof}
  It follows from Lemma~\ref{lem:urt_class} that (a) $\Rightarrow$
  (c), and we clearly have (d) $\Rightarrow$ (b) and (c) $\Rightarrow$
  (e) $\Rightarrow$ (a).  We will prove the implications (b)
  $\Rightarrow$ (a) $\Rightarrow$ (d).

  Assume that $T$ is a unique factorization target and let $T' \in
  [T]$ be any tableau of anti-straight shape.  Define the straight
  shapes $\la = \sh(T)$ and $\mu = w_X.\sh(T')$.  Then
  Corollary~\ref{cor:lrrule} and Theorem~\ref{thm:symmetry} imply that
  $c^\mu_{\la,\emptyset} = c^{\La_X}_{\la,\mu^\vee} > 0$, so we must
  have $\mu = \la$.  Furthermore, since $c^\la_{\la,\emptyset} = 1$ we
  deduce that $T'$ is the only tableau of anti-straight shape in
  $[T]$.  This proves the implication (a) $\Rightarrow$ (d).

  Assume that (b) holds and let $T'$ be the unique anti-rectification
  of $T$.  Set $\la = \sh(T)$ and let $U$ be any unique rectification
  target of shape $\la^\vee$.  Since $c^{\La_X}_{\la,\la^\vee} =
  c^\la_{\la,\emptyset} = 1$, we deduce from Corollary~\ref{cor:lrUT0}
  that $T'$ has shape $\La_X/\la^\vee$.  Let $S$ be any rectification
  of $T'$.  Then $S = \jdt_Z(T')$ for some increasing tableau $Z$ of
  shape $\la^\vee$.  To show that $T$ is a unique rectification
  target, it suffices to prove that $S=T$.

  Choose $a \in \N$ such that $Z$ has values in the interval $[-a,a]$.
  If $Z'$ is any increasing tableau of shape $\La_X/\la$, then since
  $\wh\jdt_{Z'}(T) = T'$ has shape $\La_X/\la^\vee$, it follows from
  Proposition~\ref{prop:kinfusion} that $\jdt_T(Z')$ has shape
  $\la^\vee$.  Furthermore, the assignment $Z' \mapsto \jdt_T(Z')$
  defines an injective map from the set of increasing tableaux of
  shape $\La_X/\la$ with values in $[-a,a]$ to the set of increasing
  tableaux of shape $\la^\vee$ with values in $[-a,a]$.  Now since $Z'
  \mapsto w_X.Z'$ is a bijection between the same two sets, we deduce
  from the pigeonhole principle that the map $Z' \mapsto \jdt_T(Z')$
  is surjective.  In particular, we can choose $Z'$ such that
  $\jdt_T(Z') = Z$.  Since $\wh\jdt_{Z'}(T) = T'$, another application
  of Proposition~\ref{prop:kinfusion} shows that $T = \jdt_Z(T') = S$,
  as required.
\end{proof}

\subsection{Minimal and maximal increasing tableaux}\label{sec:minmax}

For any skew shape $\nu/\la$, define the {\em minimal increasing
  tableau\/} $M_{\nu/\la}$ to be the unique increasing tableau that
fills the boxes of $\nu/\la$ with the smallest possible positive
integers.  More precisely, $M_{\nu/\la}(\al)$ is the maximal
cardinality of a totally ordered subset of $\nu/\la$ with $\al$ as its
largest element.  Define the {\em maximal increasing tableau\/} of
shape $\nu/\la$ by $\wh M_{\nu/\la} = w_X.M_{\la^\vee/\nu^\vee}$.
This tableau fills the boxes of $\nu/\la$ with the largest possible
{\em negative\/} integers.

\begin{example}
  Let $X = \Gr(m,n)$ be a Grassmann variety of type A and set $\theta
  = (9,7,6,6,4)/(5,3,2)$.  Then we have
  \[
  M_\theta = \text{\SMALL$\tableau{12}{&&&&&1&2&3&4\\ &&&1&2&3&4\\
    &&1&2&3&4\\ 1&2&3&4&5&6\\ 2&3&4&5}$}
  \hspace{5mm} \text{and} \hspace{5mm}
  \wh M_\theta = \text{\SMALL$\tableau{12}{&&&&&-4&-3&-2&-1\\ &&&-5&-4&-3&-1\\
    &&-5&-4&-3&-2\\ -6&-5&-4&-3&-2&-1\\ -4&-3&-2&-1}$} \ .\smallskip
  \]
\end{example}

The set of all minimal increasing tableaux is not closed under jeu de
taquin slides.  However, the rectification of any minimal increasing
tableau is a minimal increasing tableau.

\begin{thm}\label{thm:mininc_rect}
  Let $\nu/\la \subset \La_X$ be any skew shape.  Then $M_{\nu/\la}$
  has a unique rectification, which is the minimal increasing tableau
  $M_\mu$ of some straight shape $\mu$.
\end{thm}
\begin{proof}
  By Theorem~\ref{thm:mininc_urt} and Lemma~\ref{lem:urt_class} it is
  enough to show that the jeu de taquin class $[M_{\nu/\la}]$ contains
  at least one tableau of the form $M_\mu$.  This is proved in
  Theorem~\ref{thm:mininc_rect_A} if $X$ is a Grassmannian of type A,
  and in Corollary~\ref{cor:mininc_rect_B} if $X$ is a maximal
  orthogonal Grassmannian.  The result is an easy exercise if $X$ is a
  quadric hypersurface, and it has been verified by computer when $X$
  is the Cayley plane or the Freudenthal variety.
\end{proof}

\begin{cor}\label{cor:mininc_antirect}
  Let $\la \subset \La_X$ be any straight shape.  Then the unique
  anti-rectifica\-tion of $M_\la$ is $M_{w_X.\la}$.
\end{cor}
\begin{proof}
  Let $M_\mu$ be the rectification of $M_{w_X.\la}$.  Then
  $c^\la_{\mu,\emptyset} = c^{\La_X}_{\mu,\la^\vee} > 0$, so we must
  have $\mu = \la$.  Theorem~\ref{thm:urt_conds} shows that
  $M_{w_X.\la}$ is the only anti-rectification of $M_\la$.
\end{proof}

\begin{cor}\label{cor:maxinc_urt}
  If $\la$ is a straight shape, then $\wh M_\la$ is a unique
  rectification target.
\end{cor}
\begin{proof}
  Since the action of $w_X$ on tableaux is compatible with jeu de
  taquin slides, we deduce from Corollary~\ref{cor:mininc_antirect}
  that $w_X.[M_\la] = w_X.[M_{w_X.\la}] = [w_X.M_{w_X.\la}] = [\wh
  M_\la]$.  The result therefore follows from
  Theorem~\ref{thm:urt_conds} and the fact that $M_\la$ is a unique
  rectification target.
\end{proof}

Given an increasing tableau $T$ of shape $\nu/\la$, define the {\em
  greedy rectification\/} $\rect(T)$ as follows.  If $\la=\emptyset$
then set $\rect(T) = T$.  Otherwise define $\rect(T) =
\rect(\jdt_C(T))$, where $C$ is the set of all maximal boxes in $\la$.

\begin{cor}\label{cor:lr_greedy}
  Let $\la$, $\mu$, and $\nu$ be straight shapes in $\La_X$ and let
  $T_0$ be any increasing tableau of shape $\mu$.  Then
  $c^\nu_{\la,\mu}$ is the number of increasing tableaux $T$ of shape
  $\nu/\la$ for which $\rect(T) = T_0$.
\end{cor}
\begin{proof}
  This follows from corollaries \ref{cor:lrUT0} and
  \ref{cor:maxinc_urt} because $\rect(T) = \jdt_{\wh M_\la}(T)$.
\end{proof}


\section{$K$-Knuth equivalence of hook-closed
  tableaux}\label{sec:kknuth}

In this section we introduce our main technical tool, the {\em
  $K$-Knuth equivalence relation\/} on words of integers.  This
relation governs jeu de taquin equivalence of increasing tableaux for
Grassmannians of type A and will be used to show that minimal
increasing tableaux for these varieties are unique rectification
targets.  A weak version of $K$-Knuth equivalence will be defined in
section \ref{sec:weak_kknuth}, which plays a similar role for maximal
orthogonal Grassmannians.

Given two boxes $\al_1 = [i_1,j_1]$ and $\al_2 = [i_2,j_2]$ in $\N^2$,
we will say that $\al_1$ is {\em north of\/} $\al_2$ if $i_1 \leq
i_2$, and we will say that $\al_1$ is {\em strictly north of\/}
$\al_2$ if $i_1 < i_2$.  Similar terminology will be used for the
other directions, where the boxes in $\N^2$ are arranged as in section
\ref{sec:comin}.  {\em North-east of\/} means north of and east of,
and {\em strictly north-east\/} means strictly north and strictly
east.  Notice that $\al_1$ is north-west of $\al_2$ if and only if
$\al_1 \leq \al_2$.

A finite subset $\theta \subset \N^2$ will be called a {\em
  hook-closed shape\/} if, whenever $\theta$ contains two boxes
$\al_1=[i_1,j_1]$ and $\al_2=[i_2,j_2]$ such that $\al_1\leq\al_2$,
$\theta$ also contains the {\em north-east hook\/} spanned by $\al_1$ and
$\al_2$.  More precisely, $\theta$ must contain all boxes $[i_1,c]$ for
$j_1 \leq c \leq j_2$ and all boxes $[r,j_2]$ for $i_1 \leq r \leq
i_2$.
\[
\tableau{14}{{\al_1}&{}&{}&{}&{}&{}\\&&&&&{}\\&&&&&{}\\&&&&&{\al_2}}
\]
Examples of hook-closed shapes include skew shapes for Grassmannians
of type A and maximal orthogonal Grassmannians.

Let $\theta$ be a hook-closed shape.  A {\em weakly increasing
  tableau\/} $T$ of shape $\theta$ is a map $T : \theta \to \Z$ such
that all rows of $T$ are weakly increasing from left to right and all
columns of $T$ are weakly increasing from top to bottom.  The tableau
$T$ is {\em strictly increasing\/} (or simply {\em increasing\/}) if
all rows and columns of $T$ are strictly increasing.

\begin{defn}\label{defn:readword}
  Let $T$ be a weakly increasing tableaux of hook-closed shape.  A
  {\em reading word\/} of $T$ is any word listing the boxes of $T$ in
  any order for which (i) each box $\al$ appears before all boxes
  north-east of $\al$, (ii) if a box $\al$ is equal to the box
  immediately above $\al$, then $\al$ appears before all boxes
  strictly north of $\al$, and (iii) if a box $\al$ is equal to the
  box immediately to the right of $\al$, then $\al$ appears before all
  boxes strictly east of $\al$.
\end{defn}

Notice that some weakly increasing tableaux have no reading words.

\begin{example}
  The following is a weakly increasing tableau of hook-closed shape.
  \[
  \tableau{14}{
    &   &   &   &   &   &   &   &{2}\\
    &   &   &   &   &   &   &{1}&{2}\\
    &   &   &   &   &{2}&{2}\\
    &   &   &   &   &   &{3}\\
    &   &{1}&{2}&{4}\\
    {1}&{2}&{3}&{3}\\
    &{2}&{3}&{4}\\
    &   &   &{5}
  }
  \]
  This tableau has the following four reading words:
  \[
  \begin{split}
    &(2,3,1,2,3,1,5,4,3,2,4,2,3,2,1,2,2) \ , \ 
     (2,3,1,2,3,5,1,4,3,2,4,2,3,2,1,2,2) \\
    &(2,3,1,2,3,5,4,1,3,2,4,2,3,2,1,2,2) \ , \ 
     (2,3,1,2,3,5,4,3,1,2,4,2,3,2,1,2,2) \ .
  \end{split}
\]
\end{example}

\begin{defn}\label{defn:kknuth}
  Define the {\em $K$-Knuth equivalence relation\/} on words of
  integers, denoted by $\equiv$, as the symmetric transitive closure
  of the following basic relations.  For any words of integers $\bu =
  (u_1,\dots,u_s)$ and $\bv = (v_1,\dots,v_t)$, and $a,b,c \in \Z$ we
  set
  \[
  \begin{split}
    (\bu,a,a,\bv) &\equiv (\bu,a,\bv) \,, \\
    (\bu,a,b,a,\bv) &\equiv (\bu,b,a,b,\bv) \,, \\
    (\bu,a,b,c,\bv) &\equiv (\bu,a,c,b,\bv) \ \ \ 
    \text{if $b < a < c$\,,} \ \ \text{and} \\
    (\bu,a,b,c,\bv) &\equiv (\bu,b,a,c,\bv) \ \ \ 
    \text{if $a < c < b$\,.}
  \end{split}
  \]
\end{defn}

\begin{lemma}\label{lem:tabclass}
  Let $T$ be a weakly increasing tableau of hook-closed shape.  Then
  all reading words of\/ $T$ are $K$-Knuth equivalent.
\end{lemma}
\begin{proof}
  Let $\bu$ and $\bv$ be reading words of $T$.  We prove that $\bu
  \equiv \bv$ by induction on the number of boxes in $T$.  Write $\bu
  = (x,\bu')$ and $\bv = (y,\bv')$, where $x$ and $y$ are the first
  integers in the words.  We will identify $x$ and $y$ with the boxes
  of $T$ that they came from.  We may assume that $x$ is strictly
  north-west of $y$.  Since the shape of $T$ is hook-closed, the
  tableau $T \ssm \{x,y\}$, obtained by removing the boxes $x$ and $y$
  from $T$, has a south-west corner $z$, such that $z$ is south-east
  of $x$ and north-west of $y$.  We must have $x < z < y$, since if
  one of these inequalities is not strict, then either $\bv$ or $\bu$
  is not a valid reading word of $T$.  Let $\bw$ be any reading word
  of $T \ssm \{x,y,z\}$; for example, such a word can be obtained by
  removing the first occurrences of $x$, $y$, and $z$ from $\bu$.
  Since $\bu'$ and $(y,z,\bw)$ are reading words for $T \ssm \{x\}$,
  it follows by induction that $\bu' \equiv (y,z,\bw)$, and since
  $\bv'$ and $(x,z,\bw)$ are reading words for $T \ssm \{y\}$ we
  similarly obtain $\bv' \equiv (x,z,\bw)$.  We deduce that $\bu =
  (x,\bu') \equiv (x,y,z,\bw) \equiv (y,x,z,\bw) \equiv (y,\bv') =
  \bv$, as required.
\end{proof}

We will say that two weakly increasing tableaux $T_1$ and $T_2$ are
$K$-Knuth equivalent, written $T_1 \equiv T_2$, if a reading word of
$T_1$ is $K$-Knuth equivalent to a reading word of $T_2$.  This
defines an equivalence relation on the set of all weakly increasing
tableaux of hook-closed shape that admit reading words.

\begin{lemma}\label{lem:interval}
  Let $[a,b]$ be an integer interval.\smallskip

  \noin{\rm(a)} Let $\bw_1$ and $\bw_2$ be $K$-Knuth equivalent words.
  For $i=1,2$, let $\bw'_i$ be the word obtained from $\bw_i$ by
  skipping all integers not contained in the interval $[a,b]$.  Then
  $\bw'_1$ and $\bw'_2$ are $K$-Knuth equivalent words.\smallskip

  \noin{\rm(b)} Let $T_1$ and $T_2$ be $K$-Knuth equivalent weakly
  increasing tableaux of hook-closed shapes.  For $i=1,2$, let
  $T_i|_{[a,b]}$ denote the tableau obtained from $T_i$ by removing
  all boxes not in $[a,b]$.  Then $T_1|_{[a,b]}$ and $T_2|_{[a,b]}$
  are $K$-Knuth equivalent weakly increasing tableaux of hook-closed
  shapes.
\end{lemma}
\begin{proof}
  This is immediate from the definitions.
\end{proof}

\begin{defn}
  Let $\theta$ be a hook-closed shape.  A {\em dotted increasing
    tableau\/} of shape $\theta$ is a map $T_0 : \theta \to \Z \cup
  \{\bullet\}$ such that there exists an integer $k$ with the property
  that $T_0$ becomes a strictly increasing tableau (of rational
  numbers) if all dots in $T_0$ are replaced with the number
  $k+\frac{1}{2}$.  A {\em resolution\/} of a dotted increasing
  tableau $T_0$ is any tableau obtained by replacing each dot in $T_0$
  with either the maximum of the boxes immediately above and to the
  left of the dot, or the minimum of the boxes immediately below and
  to the right of the dot.  If there are no boxes immediately above or
  to the left of a dot, or if there are no boxes immediately below or
  to the right of a dot, then the box containing the dot may be
  removed.
\end{defn}

\begin{example}
  The following dotted increasing tableau is displayed next to one of
  its resolutions.
  \[
  \tableau{14}{
    & & & &1&3&\bullet&8\\
    & &2&3&4&6&9\\
    &1&3&5&7&\bullet\\
    &2&\bullet&8&9\\
    2&\bullet&8&9 
  }
  \ \ \ \ 
  \tableau{14}{
    & & & &1&3&8&8\\
    & &2&3&4&6&9\\
    &1&3&5&7\\
    &2&8&8&9\\
    2&2&8&9 
  }
  \]
\end{example}\smallskip

\begin{lemma}\label{lem:readword}
  If $T$ is a resolution of a dotted increasing tableau, then $T$ has
  at least one reading word.
\end{lemma}
\begin{proof}
  Let $T_0$ be a dotted increasing tableau that has $T$ as a
  resolution.  If $T_0$ contains no dots, then conditions (ii) and
  (iii) of Definition~\ref{defn:readword} are vacuous and the lemma is
  clear.  Otherwise let $\al$ be the south-west most box of $T_0$ that
  contains a dot, and let $U$ (respectively $U_0$) be the tableau of
  boxes in $T$ (respectively $T_0$) that are strictly north or
  strictly east of $\al$.  Since $U$ is a resolution of $U_0$, it
  follows by induction on the number of dots in $T_0$ that $U$ admits
  a reading word $\bu$.  We also let $V$ be the tableau of boxes in
  $T$ that are strictly south-west of $\al$, and let $\bb =
  (b_t,b_{t-1},\dots,b_1)$ be the word listing the boxes to the left of
  $\al$ and $\bc = (c_m,c_{m-1},\dots,c_1)$ the word of boxes below
  $\al$.
  \[
  T \ = \ \ 
  \tableau{14}{
    &&&&[lt]&[t]&[tr]\\
    &&[lh]&[h]&&&[br]\\
    [lh]{b_t}&[h]{\,\cdots}&[hr]{b_1}&
    [a]{{\scriptstyle T\!(\!\al\!)}}&[l]&[r]\\
    [lh]&[t]&[tr]&[tv]{c_1}&[l]&[br]\\
    &[lb]&[b]&[v]{\raisebox{6pt}{\vdots}}&{}\\
    &&&[vb]{c_m}&{}
  }
  \hspace{-11mm}\raisebox{6mm}{$U$}
  \hspace{-18mm}\raisebox{-5mm}{$V$}\medskip
  \]
  Since $T_0$ is a dotted increasing tableau, it follows that $b_1 <
  c_1$.  Let $\bv$ be a reading word of the strictly increasing
  tableau $V$.  If $b_1 = T(\al)$, then $(\bv, \bb, \bc, T(\al), \bu)$
  is a reading word for $T$, and otherwise $(\bv, \bc, \bb, T(\al),
  \bu)$ is a reading word.
\end{proof}

The following result will be used to show that $K$-theoretic jeu de
taquin slides preserve $K$-Knuth equivalence in type A.

\begin{thm}\label{thm:resolve}
  Let $T_0$ be a dotted increasing tableau of hook-closed shape.
  Assume that every dot in $T_0$ is situated to the right of an
  integer and above another integer.  Then all resolutions of $T_0$
  are $K$-Knuth equivalent.
\end{thm}
\begin{proof}
  Let $T'$ and $T''$ be resolutions of $T_0$.  If $T_0$ contains no
  dots, then $T' = T''$ and there is nothing to prove.  Otherwise let
  $\al$ be the south-west most box of $T_0$ that contains a dot.  By
  Lemma~\ref{lem:interval}(b) and the conditions of the theorem, we
  may assume that $\al$ is surrounded by boxes, except that the box
  directly south-west of $\al$ may be missing.  In the following
  picture the box $\al$ is marked with a dot and the box directly
  north-east of $\al$ is labeled $y$.  The picture includes all boxes
  to the left of (and in the same rows as) $\al$ and $y$, and all
  boxes below (and in the same columns as) $\al$ and $y$.  The numbers
  $s$, $t$, $m$, $n$ are greater than or equal to one.
  \[
  \tableau{14}{
    a_s & \,\cdots & \,\cdots & a_1 & a_0 & y \\
    & b_t & \,\cdots & b_1 & \bullet & d_0 \\
    & & & & c_1 & d_1 \\
    & & & & \raisebox{6pt}{\vdots} & \raisebox{6pt}{\vdots} \\
    & & & & \raisebox{6pt}{\vdots} & d_n \\
    & & & & c_m}
  \]

  Let $U_0$ be the tableau of boxes in $T_0$ that are strictly north
  or strictly east of $\al$; this includes the row and column of $y$.
  It follows by induction on the number of dots in $T_0$ that all
  resolutions of $U_0$ are $K$-Knuth equivalent.  Since both $T'$ and
  $T''$ have reading words that terminate in words for resolutions of
  $U_0$, we can assume that $T'$ and $T''$ put the same integers in
  all boxes except $\al$.  Furthermore, if $T_0$ contains a dot in the
  same row as $y$, then we may assume that $T'$ and $T''$ replace this
  dot with the maximum of the boxes above and to the left of the dot.
  Similarly, if $T_0$ contains a dot in the same column as $y$ but
  strictly north of $y$, then we may assume that $T'$ and $T''$
  replace this dot with the minimum of the values below and to the
  right of the dot.  These assumptions imply that $a_0 \leq y < d_0$,
  and also that $T'$ and $T''$ have reading words that start with a
  reading word of the tableau of boxes strictly south-west of $\al$,
  continue with reading words of the resolutions of the displayed
  boxes, and terminate in a reading word for the boxes strictly north
  or strictly east of $y$.  Write $\ba = (a_s,\dots,a_0)$, $\bb =
  (b_t,\dots,b_1)$, $\bc' = (c_m,\dots,c_2)$, $\bc = (\bc',c_1)$,
  $\bd' = (d_n,\dots,d_2)$, and $\bd = (\bd',d_1,d_0)$.  There are
  three cases.\smallskip

  \noin{\bf Case 1:} Assume that $T'(\al)=b_1$ and $T''(\al)=c_1$.
  Then we can choose reading words of $T'$ and $T''$ such that the
  word of $T''$ is obtained from the word of $T'$ if the subword
  $(\bb, \bc, b_1)$ is replaced with $(\bc, \bb, c_1)$, so it is
  enough to show that $(\bb, \bc, b_1) \equiv (\bc, \bb, c_1)$.
  Notice that $(b_1,\bc) \equiv (\bc',b_1,c_1)$ and $(b_j,\bc,b_{j-1})
  \equiv (\bc,b_j,b_{j-1})$ for $2 \leq j \leq t$.  The result follows
  because
  \[
  \begin{split}
    (\bb,\bc,b_1) &\equiv
    (b_t,\dots,b_2,\bc',b_1,c_1,b_1) \equiv
    (b_t,\dots,b_2,\bc,b_1,c_1) \\ 
    &\equiv 
    (b_t,\dots,b_3,\bc,b_2,b_1,c_1) \equiv
    \dots \equiv
    (\bc,\bb,c_1) \,.
  \end{split}
  \]

  \noin{\bf Case 2:} Assume that $T'(\al)=b_1$ and $T''(\al)=d_0$.
  (By symmetry this case also takes care of situations with
  $T'(\al)=a_0$ and $T''(\al)=c_1$.)  We may assume that Case 1 does
  not apply, so that $d_0 < c_1$.  In this case we can choose reading
  words of $T'$ and $T''$ such that the word of $T''$ is obtained from
  the word of $T'$ if the subword $(b_1,\bc,b_1,\bd)$ is replaced with
  $(b_1,\bc,d_0,\bd)$.  These subwords are $K$-Knuth equivalent
  because
  \[
  \begin{split}
    &(b_1,\bc,b_1,\bd) 
    \equiv (\bc',b_1,c_1,b_1,\bd)
    \equiv (\bc,b_1,c_1,\bd)
    \equiv (\bc,b_1,\bd',c_1,d_1,d_0)
    \\ &\ \ \ \ \ 
    \equiv (\bc,b_1,\bd',c_1,d_0,d_1)
    \equiv (\bc,b_1,\bd',c_1,d_0,d_1,d_1)
    \equiv (\bc,b_1,\bd',c_1,d_1,d_0,d_1)
    \\ &\ \ \ \ \ 
    \equiv (\bc,b_1,c_1,\bd,d_1)
    \equiv (\bc,b_1,c_1,\bd',d_0,d_1,d_0)
    \equiv (\bc,b_1,c_1,d_0,\bd)
    \\ &\ \ \ \ \ 
    \equiv (\bc,c_1,b_1,d_0,\bd)
    \equiv (\bc,b_1,d_0,\bd)
    \equiv (b_1,\bc,d_0,\bd) \,.
  \end{split}
  \]
  
  \noin{\bf Case 3:} Assume that $T'(\al)=a_0$ and $T''(\al)=d_0$.  We
  may assume that Case 1 and Case 2 do not apply, so that $b_1 < a_0
  \leq y < d_0 < c_1$.  We can choose reading words of $T'$ and $T''$
  such that the word of $T''$ is obtained from the word of $T'$ if the
  subword $(b_1,c_1,a_0,\ba,\bd,y)$ is replaced with
  $(b_1,c_1,d_0,\ba,\bd,y)$.  If $y=a_0$ then it follows from Case 1
  that $(\ba,\bd,y) \equiv (\bd,\ba,y')$, where we set $y'=d_0$.
  Otherwise Lemma~\ref{lem:tabclass} implies that $(\ba,\bd,y) \equiv
  (\bd,\ba,y')$, where we set $y'=y$.  Notice also that Case 2 applied
  to subsets of the displayed boxes implies that $(c_1,b_1,a_0,\ba)
  \equiv (c_1,b_1,c_1,\ba)$ and $(b_1,c_1,b_1,\bd) \equiv
  (b_1,c_1,d_0,\bd)$.  We finally obtain
  \[
  \begin{split}
  (b_1,c_1,a_0,\ba,\bd,y)
  &\equiv (c_1,b_1,a_0,\ba,\bd,y)
  \equiv (c_1,b_1,c_1,\ba,\bd,y)
  \equiv (b_1,c_1,b_1,\bd,\ba,y') \\
  &\equiv (b_1,c_1,d_0,\bd,\ba,y')
  \equiv (b_1,c_1,d_0,\ba,\bd,y) \,.
  \end{split}
  \]
  This completes the proof.
\end{proof}

\begin{remark}\label{rmk:kknuth_difficult}
  In contrast to the ordinary Knuth relation, it appears to be
  difficult to determine if two given words are $K$-Knuth equivalent.
  The problem is that in order to transform one word into another
  using the basic relations of Definition~\ref{defn:kknuth}, it may be
  necessary to go through a sequence of longer words.  For example,
  this happens in Case 2 of the proof of Theorem~\ref{thm:resolve}.
  It would be interesting to know if this problem is decidable in
  general.
\end{remark}


\section{Grassmannians of type A}\label{sec:grass}

\subsection{Tableaux of type A}

In this section we work with the partially ordered set $\La = \N^2$,
where the order is given by $[r_1,c_1] \leq [r_2,c_2]$ if and only if
$r_1\leq r_2$ and $c_1\leq c_2$.  A straight shape $\la \subset \N^2$
is the same as a {\em Young diagram\/} and can be identified with the
partition $(\la_1 \geq \dots \geq \la_\ell)$ where $\ell$ is the
number of rows in $\la$ and $\la_i$ is the
number of boxes in row $i$.  Any tableau defined on a skew
shape in $\N^2$ will be called a {\em tableau of type A}.  If $X =
\Gr(m,m+k)$ is a Grassmann variety of type A, then we identify $\La_X$
with the boxes $[r,c] \in \N^2$ for which $1 \leq r \leq m$ and $1
\leq c \leq k$.  Every tableau for $\La_X$ is then a tableau of type
A.

Let $T$ and $T'$ be increasing tableaux of type A.  Then $T$ and $T'$
are jeu de taquin equivalent for $\N^2$ if and only if $T$ and $T'$
are jeu de taquin equivalent for $\La_X$ for all Grassmannians $X =
\Gr(m,m+k)$ given by sufficiently large integers $m$ and $k$.
Furthermore, an increasing tableau $U$ of straight shape is a unique
rectification target for $\N^2$ if and only if $U$ is a unique
rectification target for $\La_X$ for all Grassmannians $X$ satisfying
$\sh(U) \subset \La_X$.  In this section we study unique rectification
targets for $\N^2$.

\begin{lemma}\label{lem:resolve_A}
  Let $T$ be a dotted increasing tableau of type A.  Then all
  resolutions of $T$ are $K$-Knuth equivalent.
\end{lemma}
\begin{proof}
  Using that the shape of $T$ is a skew Young diagram and
  Lemma~\ref{lem:interval}(b), we may assume that every dot in $T$ is
  situated to the right of an integer and above another integer.  The
  lemma therefore follows from Theorem~\ref{thm:resolve}.
\end{proof}

Recall that the {\em row word\/} of a tableau $T$ is the reading word
obtained by reading the rows of $T$ from left to right, starting with
the bottom row.

\begin{thm}\label{thm:jdt_kknuth}
  Let $T$ and $T'$ be increasing tableaux of type A.  Then $T$ and
  $T'$ are $K$-Knuth equivalent if and only if $T$ and $T'$ are jeu de
  taquin equivalent for $\N^2$.  
\end{thm}
\begin{proof}
  For any word $\bu = (u_1,\dots,u_p) \in \Z^p$ we define an
  increasing tableau $T(\bu)$ whose shape is the antidiagonal
  $\{[p+1-j,j] \mid 1 \leq j \leq p \}$ and whose values are given by
  $T([p+1-j,j]) = u_j$.  By examining each of the basic relations of
  Definition~\ref{defn:kknuth}, it follows that if $\bu \equiv \bu'$,
  then $T(\bu')$ can be obtained from $T(\bu)$ by a sequence of
  forward and reverse jeu de taquin slides.  Furthermore, if $\bu$ is
  the row word of an increasing tableau $S$, then $T(\bu)$ can be
  obtained from $S$ by a sequence of slides.  The implication ``only
  if'' of the theorem follows from these observations.

  For the other implication we must show that $K$-Knuth equivalence is
  preserved by arbitrary jeu de taquin slides.  Assume that $T' =
  \jdt_C(T)$ and that $T$ and $T'$ have values in the interval
  $[a,b]$.  Define a sequence of dotted increasing tableaux $T_a,
  T_{a+1}, \dots, T_{b+1}$ by $T_a = [C \to \bullet] \cup T$ and
  $T_{i+1} = \swap_{i,\bullet}(T_i)$ for $a \leq i \leq b$.  Then $T$
  is a resolution of $T_a$ and $T' = T_{b+1}|_\Z$ is a resolution of
  $T_{b+1}$.  Let $T_i'$ be the resolution of $T_i$ obtained by
  replacing each dot with the minimum of the boxes below and to the
  right of the dot.  Since $T_i'$ is also a resolution of $T_{i+1}$,
  it follows from Lemma~\ref{lem:resolve_A} that $T'_i \equiv
  T'_{i+1}$.  We deduce that $T \equiv T_a' \equiv T_{a+1}' \equiv \dots
  \equiv T_{b+1}' = T'$.
\end{proof}

\begin{cor}\label{cor:urt_kknuth}
  Let $T$ be an increasing tableau of straight shape of type A.  The
  following are equivalent.  {\rm(a)} $T$ is a unique rectification
  target for $\N^2$.  {\rm(b)} If $T'$ is any increasing tableau of
  straight shape such that $T' \equiv T$, then $T' = T$.
\end{cor}
\begin{proof}
  This follows from Theorem~\ref{thm:jdt_kknuth} and
  Lemma~\ref{lem:urt_class}.
\end{proof}

\subsection{The Hecke permutation}

Theorem~\ref{thm:jdt_kknuth} implies that the $K$-Knuth equivalence
class of an increasing tableau is an invariant under jeu de taquin
slides, in fact the finest such invariant.  However, as indicated in
Remark~\ref{rmk:kknuth_difficult}, this invariant may be difficult to
work with.  We therefore define a coarser invariant which is much
easier to understand but still able to identify the jeu de taquin
class $[U]$ for many important unique rectification targets $U$.

Let $\Sigma$ be the group of bijective maps $w : \Z \to \Z$ such that
$w(x)=x$ for all but finitely many integers $x \in \Z$.  The elements
of $\Sigma$ will be called {\em permutations}.  For any integer $i \in
\Z$ we let $s_i = (i,i+1) \in \Sigma$ be the simple transposition that
interchanges $i$ and $i+1$.  We will consider $\Sigma$ as a Coxeter
group generated by these simple transpositions.  We need the {\em
  Hecke product\/} on $\Sigma$, which can be defined as follows.  For
any permutation $u \in \Sigma$ and simple transposition $s_i$, we
define
\[
u \cdot s_i = \begin{cases}
  u s_i & \text{if $\ell(u s_i) > \ell(u)$;} \\
  u & \text{otherwise.}
\end{cases}
\]
Given an additional permutation $v \in \Sigma$, we then set
\[
u \cdot v = u \cdot s_{i_1} \cdot s_{i_2} \cdot \ldots \cdot
s_{i_\ell}
\]
where $v = s_{i_1} s_{i_2} \cdots s_{i_\ell}$ is any reduced
expression for $v$, and the simple transpositions are multiplied to
$u$ in left to right order.  This is independent of the chosen reduced
expression and defines an associative monoid product on $\Sigma$.  The
product $u \cdot v$ is called {\em reduced\/} if $\ell(u \cdot v) =
\ell(u) + \ell(v)$, which is true if and only if $u \cdot v = uv$,
i.e.\ the Hecke product agrees with the usual product in $\Sigma$.
Proofs of these facts can be found in e.g.\ \cite[\S
3]{buch.mihalcea:curve}.

Given a word $\ba = (a_1,\dots,a_k)$ of integers, define the
permutation $w(\ba) = s_{a_1} \cdot s_{a_2} \cdot \ldots \cdot
s_{a_k}$.  If $\ba$ is a reading word of an increasing tableau $T$ of
type A, then we also write $w(T) = w(\ba)$.  This permutation is
called the {\em Hecke permutation\/} of $T$ and was used in
\cite{buch.kresch.ea:stable} for increasing tableaux of straight
shape.  We will show that this permutation is independent of the
chosen reading word and invariant under jeu de taquin slides.  This is
done by noting that the relations of the Hecke monoid are a weakening
of the $K$-Knuth relations.

\begin{defn}\label{defn:heckerel}
  Define the {\em Hecke equivalence relation\/} on words of integers,
  denoted $\approx$, to be the symmetric transitive closure of the
  following basic relations.  For any words $\bu$ and $\bv$, and
  integers $a,b$ we set
  \[
  \begin{split}
    (\bu,a,a,\bv) &\approx (\bu,a,\bv) \,, \\
    (\bu,a,b,a,\bv) &\approx (\bu,b,a,b,\bv) \,, \\
    (\bu,a,b,\bv) &\approx (\bu,b,a,\bv) \ \ \ \ \ \text{if $|a-b| \geq 2$.}
  \end{split}
  \]
\end{defn}

This relation governs the Hecke monoid in the sense that $\bu \approx
\bu'$ if and only if $w(\bu) = w(\bu')$ for any words of integers
$\bu$ and $\bu'$.  On the other hand, notice that the $K$-Knuth
relation $\bu \equiv \bu'$ implies the Hecke relation $\bu \approx
\bu'$.  In fact, the two last basic relations of
Definition~\ref{defn:kknuth} are both special cases of the last
relation given in Definition~\ref{defn:heckerel}.  We record the
following consequence of Theorem~\ref{thm:jdt_kknuth}.

\begin{cor}\label{cor:jdt_hecke}
  The Hecke permutation $w(T)$ of an increasing tableau of type A is
  invariant under jeu de taquin slides.
\end{cor}

\begin{thm}\label{thm:mininc_urt_A}
  Let $\la$ be a Young diagram.  Then the minimal increasing tableau
  $M_\la$ is a unique rectification target for $\N^2$.  Furthermore,
  an increasing tableau $T$ of type A rectifies to $M_\la$ if and only
  if $w(T) = w(M_\la)$.
\end{thm}
\begin{proof}
  In view of Corollary~\ref{cor:jdt_hecke} it is enough to show that,
  if $T$ is any increasing tableau of straight shape such that $w(T) =
  w(M_\la)$, then $T = M_\la$.

  Set $w = w(T) = w(M_\la)$.  Using that $w$ is the Hecke permutation
  defined by the row word of $M_\la$, we obtain that $w^{-1}(1) =
  \la_1+1$.  Since $w$ is also the Hecke permutation given by the row
  word of $T$, this implies that the first row of $T$ must start with
  the integers from $1$ to $\la_1$, and if this row has more than
  $\la_1$ boxes, then the remaining boxes are greater than or equal to
  $\la_1+2$.

  Let $\ov\la = (\la_2,\la_3,\dots,\la_{\ell(\la)})$ be the Young
  diagram obtained by removing the first row of $\la$ and let
  $\ov{M_\la}$ be the tableau obtained by removing the first row of
  $M_\la$.  Let $\ov{T}$ denote the {\em skew\/} tableau obtained by
  removing the first $\la_1$ boxes from the top row of $T$, and let
  $T'$ be a rectification of $\ov{T}$.  We also set $u = s_1 s_2
  \cdots s_{\la_1} \in \Sigma$.  Then we have
  \[
  w(\ov{M_\la}) \cdot u = w(M_\la) = w(T) = w(\ov{T}) \cdot u = w(T')
  \cdot u \,.
  \]
  Furthermore, since all boxes of $\ov{M_\la}$ and $T'$ are greater
  than or equal to two, it follows that the Hecke products are
  reduced, hence $w(\ov{M_\la}) = w(T')$.  By induction on $\ell(\la)$
  this implies that $T' = \ov{M_\la}$.  Since all boxes in the top row
  of $\ov{M_\la}$ are strictly smaller than $\la_1+2$, we deduce that
  the top row of $T$ contains exactly $\la_1$ boxes, hence $\ov{T} =
  T' = \ov{M_\la}$.  We conclude that $T = M_\la$.
\end{proof}

Notice that Theorem~\ref{thm:mininc_urt_A} provides yet another
formulation of the Littlewood-Richardson rule for the $K$-theory of
Grassmannians.  The structure constant $c^\nu_{\la,\mu}$ is equal to
the number of increasing tableaux $T$ of shape $\nu/\la$ for which
$w(T) = w(M_\mu)$.

\subsection{The Pieri rule for $\Gamma(\N^2)$}

The following result shows that multiplication with a single-row
partition in the combinatorial $K$-theory ring $\Gamma(\N^2)$ is
compatible with Lenart's Pieri rule for the $K$-theory of
Grassmannians \cite{lenart:combinatorial}.  A {\em horizontal strip\/}
is a skew shape in $\N^2$ that contains at most one box in each
column.

\begin{cor}\label{cor:pieri_A}
  Let $\la$ be a Young diagram and let $p$ be a positive integer.
  Then we have
  \[
  G_{(p)} \cdot G_\la = \sum_\nu \binom{r(\nu/\la)-1}{|\nu/\la|-p} \,
  G_\nu \,,
  \]
  in the ring $\Gamma(\N^2)$, where the sum is over all partitions
  $\nu \supset \la$ such that $\nu/\la$ is a horizontal strip, and
  $r(\nu/\la)$ is the number of non-empty rows of $\nu/\la$.
\end{cor}
\begin{proof}
  The coefficient of $G_\nu$ in the product $G_{(p)}\cdot G_\la$ is
  equal to the number of increasing tableaux $T$ of shape $\nu/\la$
  for which $w(T) = s_1 s_2 \cdots s_p$.  The condition $w(T) = s_1
  s_2 \cdots s_p$ holds if and only if the row word of $T$ is weakly
  increasing and contains exactly the integers $\{1,2,\dots,p\}$.
  This implies that $\nu/\la$ must be a horizontal strip, in which
  case there are exactly $\binom{r(\nu/\la)-1}{|\nu/\la|-p}$ choices
  for $T$.
\end{proof}

\subsection{Numerical invariants}

We next derive two additional invariants from the $K$-Knuth class of
an increasing tableau.  These invariants were also obtained in
\cite{thomas.yong:jeu} with different methods.  If $\bu =
(u_1,\dots,u_\ell)$ is a word of integers, then a {\em subsequence\/}
of $\bu$ is any word of the form $(u_{i_1},\dots,u_{i_k})$ for indices
$i_1 < \dots < i_k$.  We let $\lis(\bu)$ denote the length of the
longest strictly increasing subsequence of $\bu$, and let $\lds(\bu)$
be the length of the longest strictly decreasing subsequence.  It
follows by inspection of the basic relations of
Definition~\ref{defn:kknuth} that, if $\bu$ and $\bu'$ are $K$-Knuth
equivalent words of integers, then $\lis(\bu) = \lis(\bu')$ and
$\lds(\bu) = \lds(\bu')$.  If $\bu$ is a reading word of an increasing
tableau $T$ of type A, we can therefore write $\lis(T) = \lis(\bu)$
and $\lds(T) = \lds(\bu)$ without ambiguity.  Notice that if $T$ has
straight shape, then $\lis(T)$ is equal to the number of columns in
$T$, and $\lds(T)$ is the number of rows in $T$.  The following result
follows from Theorem~\ref{thm:jdt_kknuth}.

\begin{cor}[Thomas and Yong]\label{cor:lis}
  Let $T$ be an increasing tableau of type A.  Then $\lis(T)$ and
  $\lds(T)$ are invariant under jeu de taquin slides.
\end{cor}

We now prove a generalization of Thomas and Yong's result that
superstandard tableaux of type A are unique rectification targets.  By
a {\em fat hook\/} we will mean a partition of the form $(a^b,c^d)$
where $a,b,c,d$ are non-negative integers with $a \geq c$.  For
example:
\[
(7^3,2^2) = (7,7,7,2,2) \ = \ 
\tableau{8}{{}&{}&{}&{}&{}&{}&{}\\{}&{}&{}&{}&{}&{}&{}\\
{}&{}&{}&{}&{}&{}&{}\\{}&{}\\{}&{}} \ .
\]
Let $\la = (a^b,c^d)$ be a fat hook, let $M_\la$ be the corresponding
minimal increasing tableau, and let $U$ be an increasing tableau of
straight shape.  We will say that $U$ {\em fits in the corner of
  $M_\la$\/} if $U$ has at most $d$ rows, at most $a-c$ columns, and
all integers contained in $U$ are strictly larger than the integers
contained in $M_\la$.  In this case we let $M_\la \cup U$ denote the
increasing tableau obtained by attaching $U$ to the corner of $M_\la$.
\[
M_\la \cup U \ = \ \ 
\tableau{8}{
[lt]&[t]&{}&{}&{}&{}&{}&[tr]\\ 
[l]&&&[b]&{}&{}&{}&[br]\\
[l]&[]M_\la&[r]&[lt]&[t]&{}&[hr]\\
[l]&&[r]&[l]&[]U&[br]\\
[lb]&[b]&[br]&[lb]&[br]
}\medskip
\]

\begin{thm}\label{thm:fathook}
  Let $\la$ be a fat hook, and let $U$ be any unique rectification
  target that fits in the corner of $M_\la$.  Then $M_\la \cup U$ is a
  unique rectification target.
\end{thm}
\begin{proof}
  Let $T$ be a tableau of straight shape that is jeu de taquin
  equivalent to $M_\la \cup U$, and let $a$ be the largest integer
  contained in $M_\la$.  Then it follows from
  Theorem~\ref{thm:mininc_urt_A} and Lemma~\ref{lem:jdt_interval}
  that $T|_{[1,a]} = M_\la$, after which Corollary~\ref{cor:lis}
  implies that $T = M_\la \cup U'$ for some increasing tableau $U'$
  that fits in the corner of $M_\la$.  Since
  Lemma~\ref{lem:jdt_interval} shows that $U'$ is jeu de taquin
  equivalent to $U$, the assumption that $U$ is a unique rectification
  target implies that $U' = U$, as required.
\end{proof}

\begin{cor}[Thomas and Yong]
  Let $\la$ be a Young diagram.  Then the superstandard tableaux
  $S_\la$ and $\wh S_\la$ of shape $\la$ are unique rectification
  targets.
\end{cor}

\begin{remark}
  It is tempting to look for generalizations of
  Theorem~\ref{thm:fathook}.  However, many natural generalizing
  statements are ruled out by the fact that the following two tableaux
  \[
  \tableau{10}{1&2&3\\2\\4} \ \ \equiv \ \ 
  \tableau{10}{1&2&3\\2&4\\4}\medskip
  \]
  are jeu de taquin equivalent, and therefore the first tableau is not
  a unique rectification target.
\end{remark}

\subsection{Stable Grothendieck polynomials}

Our methods can be applied to obtain a simple formula for the product
of any stable Grothendieck polynomial with a stable Grothendieck
polynomial given by a partition.  Let $S_n \subset \Sigma$ be the
subgroup of permutations of the integer interval $[1,n]$.  For each $w
\in S_n$, define a {\em Grothendieck polynomial\/} $\groth_w =
\groth_w(x_1,\dots,x_{n-1})$ as follows.  If $w = w_0 \in S_n$ is the
longest permutation, then we set $\groth_{w_0} = \prod_{i=1}^{n-1}
x_i^{n-i}$.  Otherwise we set
\[
\groth_w = \frac{(1+x_{i+1}) \groth_{w
    s_i}(x_1,\dots,x_n) - (1+x_i) \groth_{w
    s_i}(x_1,\dots,x_{i+1},x_i,\dots,x_n)}{x_i - x_{i+1}}
\]
for any $i$ such that $w(i) < w(i+1)$.  

Given any permutation $w \in \Sigma$ and $m \in \Z$, the {\em shifted
  permutation\/} $1^m\times w \in \Sigma$ is defined by $(1^m\times
w)(i) := w(i-m)+m$ for $i \in \Z$.  Notice that for all sufficiently
large integers $m$ we have $1^m \times w \in S_{2m}$, hence the
polynomial $\groth_{1^m\times w}$ is defined.  The {\em stable
  Grothendieck polynomial\/} for $w$ is the power series in infinitely
many variables obtained as the limit
\[
G_w \ = \ \lim_{m \to \infty} \groth_{1^m \times w} \ \in \
\Z\llbracket x_1, x_2, \ldots \rrbracket \,.
\]

The polynomials $(-1)^{\ell(w)} \groth_w(-x_1,\dots,-x_{n-1})$ were
introduced by Lascoux and Sch\"utzenberger as representatives for the
Schubert classes in the $K$-theory of the flag manifold $\Fl(\C^n)$
\cite{lascoux.schutzenberger:structure}, and stabilizations of these
polynomials were studied by Fomin and Kirillov in
\cite{fomin.kirillov:grothendieck}.  We have deviated from the
original definitions to ensure that the structure constants for
products of Grothendieck polynomials are non-negative
\cite{brion:positivity}.

If $\la = (\la_1 \geq \dots \geq \la_\ell > 0)$ is any partition, we
let $w_\la \in \Sigma$ be the corresponding Grassmannian permutation,
defined by $w_\la(i) = i+\la_{\ell+1-i}$ for $1 \leq i \leq \ell$,
$w_\la(0)=0$, and $w_\la(i) < w_\la(i+1)$ for $i \neq \ell$.  The
stable Grothendieck polynomial for $\la$ is defined by $G_\la =
G_{w_\la}$.  It was proved in
\cite[Thm.~6.3]{buch:littlewood-richardson} that every stable
Grothendieck polynomial $G_w$ can be written as a finite linear
combination of the stable polynomials indexed by partitions:
\[
G_w = \sum_\la a_{w,\la}\, G_\la \,.
\]
By using Theorem~\ref{thm:iso_groth} and
\cite[Thm.~8.1]{buch:littlewood-richardson}, we may identify each
stable polynomial $G_\la$ with the basis element of the same name in
the combinatorial $K$-theory ring $\Gamma(\N^2)$.  In this way we can
consider each stable Grothendieck polynomial $G_w$ as an element of
$\Gamma(\N^2)$.

A result of Lascoux \cite{lascoux:transition} shows that each
coefficient $a_{w,\la}$ is non-negative.  It was proved in
\cite{buch.kresch.ea:stable} that $a_{w,\la}$ is equal to the
number of increasing tableaux $T$ of shape $\la$ such that $w(T) =
w^{-1}$.  This formula has the following generalization.

\begin{cor}\label{cor:sgroth_prod}
  For any permutation $w \in \Sigma$ and Young diagram $\la$ we
  have\linebreak $G_w \cdot G_\la = \sum_\nu c^\nu_{w,\la} G_\nu$,
  where $c^\nu_{w,\la}$ is the number of increasing tableaux $T$ of
  shape $\nu/\la$ such that $w(T) = w^{-1}$.
\end{cor}
\begin{proof}
  Let $\tau = \{ T \mid w(T) = w^{-1} \}$ be the set of all increasing
  tableaux $T$ of type A for which $w(T) = w^{-1}$.  It follows from
  Corollary~\ref{cor:jdt_hecke} that $\tau$ is closed under jeu de
  taquin slides, and \cite[Thm.~1]{buch.kresch.ea:stable} shows that
  $G_w = F_\tau(1)$ with the notation of \S\ref{sec:comb_ktheory}.
  Proposition~\ref{prop:assoc} therefore shows that $G_w \cdot G_\la =
  F_\tau(G_\la)$, as required.
\end{proof}

\begin{remark}
  Given a partition $\la$, there are several ways to choose a
  permutation $w$ for which $G_w = G_\la$, including $w_\la$,
  $w(M_\la)^{-1}$, and shifts of these permutations.  For any such
  permutation $w$ there exists a unique increasing tableau $U$ of
  shape $\la$ such that $w(U) = w^{-1}$, and this tableau is a unique
  rectification target.
\end{remark}

\subsection{Rectification of minimal increasing tableaux}

Given increasing tableaux $S$ and $T$ of type A, we let $S * T$ denote
the increasing tableau obtained by attaching the north-east corner of
$S$ to the south-west corner of $T$.  We then let $S \cdot T = \rect(S
* T)$ denote the greedy rectification of this tableau (see
\S\ref{sec:minmax}).  We are mainly interested in this construction
when $S$ and $T$ have straight shapes.  For example, we have
\[
\tableau{10}{1&2\\4} * \tableau{10}{1&3\\3} \ = \  
\tableau{10}{&&1&3\\ &&3\\ 1&2\\ 4}
\ \ \ \ \text{and} \ \ \ \ 
\tableau{10}{1&2\\4} \cdot \tableau{10}{1&3\\3} \ = \ 
\tableau{10}{1&2&3\\2\\4} \ .
\]
It has been proved in \cite{thomas.yong:direct} that $S\cdot T$ is
the only rectification of $S*T$, but we will not use this fact here.
On the other hand, the product on increasing tableaux is not
associative since
\[
\left( \tableau{10}{1} \cdot \tableau{10}{1&4\\3} \right) 
\cdot \tableau{10}{2} \ = \ \tableau{10}{1&2&4\\3}
\text{ \ \ \ whereas \ \ \ }
\tableau{10}{1} \cdot \left( \tableau{10}{1&4\\3} \cdot 
\tableau{10}{2} \right) \ = \ \tableau{10}{1&2&4\\3&4} \ .
\]
However, the product is associative up to jeu de taquin equivalence.
More precisely, if $S$, $T$, and $U$ are increasing tableaux, then
\[
(S \cdot T) \cdot U \equiv S \cdot (T \cdot U) 
\]
as both sides are jeu de taquin equivalent to $S*T*U$.  Notice also
that the product of increasing tableaux satisfies the identity
\[
S^\dagger \cdot T^\dagger = (T \cdot S)^\dagger \,,
\]
where $T^\dagger$ is the conjugate of $T$, obtained by mirroring the
boxes of $T$ in the north-west to south-east diagonal.

\begin{lemma}\label{lem:mininc_prod}
  Let $M$ and $N$ be minimal increasing tableaux of straight shapes of
  type A.  Then $M \cdot N$ is a minimal increasing tableau.
\end{lemma}
\begin{proof}
  Since all minimal increasing tableaux of straight shapes are unique
  rectification targets by Theorem~\ref{thm:mininc_urt_A}, it suffices
  to show that $M*N$ is jeu de taquin equivalent to some minimal
  increasing tableau of straight shape.  We may assume that $M$ and
  $N$ are both non-empty tableaux.  We then proceed by induction on
  $\lds(M)+\lis(N)$, the number of rows in $M$ plus the number of
  columns in $N$.
  
  In the base case $\lds(M)+\lis(N) = 2$, $M = M_{(a)}$ is a minimal
  increasing tableau with a single row, and $N = M_{(1^b)}$ is a
  minimal increasing tableau with a single column.  In this case we
  leave it to the reader to check that $M \cdot N = M_{(c,1^{d-1})}$
  is a minimal increasing tableau of hook shape, where
  \[
  c = \begin{cases}
    a & \text{if $a \geq b$}\\
    a+1 & \text{if $a < b$}
  \end{cases}
  \text{ \ \ \ \ \ \ and \ \ \ \ \ \ }
  d = \begin{cases}
    b & \text{if $b \geq a$}\\
    b+1 & \text{if $b < a$.}
  \end{cases}
  \]
  
  Assume next that $\lds(M)+\lis(N) \geq 3$.  By possibly replacing
  $(M,N)$ with $(N^\dagger,M^\dagger)$, we may assume that $N$ has at
  least two columns.  Let $N'$ be the first column of $N$ and let
  $N''$ be the rest of $N$.  Then form the product $T = M \cdot N'$,
  and let $T'$ be the first column of $T$ and $T''$ the rest of $T$.
  It follows from the induction hypothesis that $T$ is a minimal
  increasing tableau.  Therefore $T''$ is minimal increasing with
  smallest element 2, by which we mean that $T''$ becomes a minimal
  increasing tableau in the usual sense if all its integers are
  decreased by one.  Since $N''$ is also minimal increasing with
  smallest element 2, and since $\lds(T'') \leq \lds(M)$, a second
  application of the induction hypothesis shows that the tableau $S =
  T'' \cdot N''$ is minimal increasing with smallest element 2.
  Notice also that
  \[
  M \cdot N = M \cdot (N' \cdot N'') \equiv (M \cdot N') \cdot N''
  = (T' \cdot T'') \cdot N'' \equiv T' \cdot (T'' \cdot N'')
  = T' \cdot S \,.
  \]
  It follows that $\lds(T') = \lds(T) = \lds(M * N') = \lds(M * N) =
  \lds(T' \cdot S) \geq \lds(S)$.  Since $T'$ is a minimal increasing
  tableau with a single column, $S$ is minimal increasing with
  smallest element 2, and $T'$ has at least as many rows as $S$, we
  deduce that the product $T' \cdot S$ is obtained by attaching $S$ to
  the right side of $T'$.  This shows that $T' \cdot S$ is a minimal
  increasing tableau and completes the proof.
\end{proof}

\begin{thm}\label{thm:mininc_rect_A}
  Let $T$ be any minimal increasing tableau of type A.  Then
  $\rect(T)$ is a minimal increasing tableau.
\end{thm}
\begin{proof}
  Let $\al=[r,c]$ be a box of $\sh(T)$ that is as far north-west as
  possible, i.e. a box for which $r+c$ is minimal.  Let $S$ be the
  tableau of boxes in $T$ that are south-east of $\al$, let $M$ be the
  tableau of boxes strictly west of $\al$, and let $N$ be the tableau
  of boxes strictly north of $\al$.
  \[
  T \ = \ \ \tableau{10}{&&&&&[lt]&[t]&[hr]\\ &&&&[lh]&[b]&[br]\\
    &&[lt]&[t]&{}&[hr]\\ & [vt]&[l]&[b]&[br]\\ [lt]&[r]&[vb]\\
    [lb]&[br]}
  \hspace{-26mm}\raisebox{-7mm}{$M$}
  \hspace{6mm}S
  \hspace{5mm}\raisebox{6mm}{$N$}
  \medskip
  \]
  The choice of $\al$ implies that $S$, $M$, and $N$ are minimal
  increasing tableaux, and $S$ has straight shape.  By induction on
  the number of boxes in $T$, it follows that $\rect(M)$ and
  $\rect(N)$ are both minimal increasing tableaux of straight shapes.
  Lemma~\ref{lem:mininc_prod} therefore implies that $(\rect(M) \cdot
  S) \cdot \rect(N)$ is a minimal increasing tableau.  The result
  follows because $T$ is jeu de taquin equivalent to this product,
  which is a unique rectification target by
  Theorem~\ref{thm:mininc_urt_A}.
\end{proof}

\begin{remark}
  It follows from Lemma~\ref{lem:mininc_prod} that the product of
  increasing tableaux of type A is associative when restricted to the
  set of all minimal increasing tableaux.  More precisely, if $S$,
  $T$, and $U$ are minimal increasing tableaux of type A, then
  $(S\cdot T)\cdot U = S\cdot (T\cdot U)$, as both sides are minimal
  increasing tableaux of straight shape, and all such tableaux are
  unique rectification targets by Theorem~\ref{thm:mininc_urt_A}.  Let
  $\cP$ denote the set of all Young diagrams.  Then the product $\la
  \cdot \mu := \sh(M_\la \cdot M_\mu)$ gives $\cP$ the structure of
  associative monoid, with unit equal to the empty Young diagram.
  Furthermore, the map $\cP \to \Sigma$ defined by $\la \mapsto
  w(M_\la)$ is a monoid homomorphism, where $\Sigma$ is equipped with
  the Hecke product.
\end{remark}

\ignore{
\begin{remark}
  FIXME: Factor sequences.
\end{remark}
}


\section{Maximal orthogonal Grassmannians}\label{sec:maxog}

\subsection{Tableaux of type B}

Let $\Delta = \{ [r,c] \in \N^2 \mid r \leq c \}$ be the set of boxes
in $\N^2$ that are on or above the diagonal.  We give $\Delta$ the
partial order inherited from $\N^2$.  A straight shape in $\Delta$ is
the same as a {\em shifted Young diagram\/} and can be identified with
the strict partition $(\la_1 > \dots > \la_\ell)$ where $\ell$ is the
number of rows in $\la$ and $\la_i$ is the number of boxes in row $i$.
Any tableau defined on a skew shape in $\Delta$ will be called a {\em
  tableau of type B}.  If $X = \OG(n,2n)$ is a maximal orthogonal
Grassmannian, then we identify $\La_X$ with the boxes $[r,c] \in
\Delta$ for which $c \leq n-1$.  Every tableau for $\La_X$ is then a
tableau of type B.  Two increasing tableaux $T$ and $T'$ of type B are
jeu de taquin equivalent for $\Delta$ if and only if they are jeu de
taquin equivalent for $\La_X$ for all $X = \OG(n,2n)$ with $n$
sufficiently large.  An increasing tableau $U$ of straight shape of
type B is a unique rectification target for $\Delta$ if and only if
$U$ is a unique rectification target for $\La_X$ for all maximal
orthogonal Grassmannians $X$ satisfying $\sh(U) \subset \La_X$.

If $\theta \subset \Delta$ is a skew shape, then we let $\theta^2
\subset \N^2$ denote the union of $\theta$ with its reflection in the
diagonal, i.e.\ $\theta^2 = \{ [r,c] \in \N^2 \mid [r,c] \in \theta
\text{ or } [c,r] \in \theta \}$.  This is a skew shape in $\N^2$.  If
$T$ is an increasing tableau of type B with $\sh(T)=\theta$, then we
let $T^2$ denote the tableau of type A of shape $\theta^2$ defined by
\[
T^2([r,c]) = \begin{cases}
  T([r,c]) & \text{if $r\leq c$,} \\
  T([c,r]) & \text{if $r\geq c$.}
\end{cases}
\]
We call $T^2$ the {\em doubling\/} of $T$.  For example, the tableau\smallskip
\[
T = \tableau{10}{
  [lh]&[t]&{}&{}&{}&[VT]2&[t]\\
  &[lb]&&[LT]1&[lT]3&[lhR]4\\
  &&[LH]2&[a]4&6&[ltBR]7\\
  &&&[LB]5&[ltBR]7\\
  &&&&[lb]\\
  &&&&&{}\\
  &&&&&&{}
}
\text{ \ \ \ has doubling \ \ \ }
T^2 = \tableau{10}{
  [lt]&[t]&{}&{}&{}&[VTb]2&[t]\\
  [l]&&&[LT]1&[at]3&[ar]4\\
  [l]&&[LT]2&[a]4&6&[BR]7\\
  [l]&[LT]1&[a]4&5&[BR]7\\
  [l]&[la]3&[a]6&[BR]7\\
  [LH]2&[ab]4&[BR]7\\
  [l]
} \ .\medskip
\]
If $T$ is an increasing tableau of type B of shape $\nu/\la$, and $C
\subset \la$ is a subset of the maximal boxes in $\la$, then it
follows immediately from the definitions that
\begin{equation}\label{eqn:doubling}
\jdt_C(T)^2 = \jdt_{C^2}(T^2) \,.
\end{equation}
The use of doubled tableaux appears to originate in Worley's thesis
\cite[\S6.3]{worley:theory}, which gives a slightly different
construction such that the number of boxes in a doubled tableau is
always twice the number of boxes in the original tableau.  A
modification of Worley's doubling was used in
\cite{clifford.thomas.ea:k-theoretic} and is shown to satisfy an
identity similar to (\ref{eqn:doubling}), with the proof requiring
slightly more work, see
\cite[Lemma~3.2]{clifford.thomas.ea:k-theoretic}.  Equation
(\ref{eqn:doubling}) has the following consequence.

\begin{prop}\label{prop:translateAB}
  {\rm(a)} If $S$ and $T$ are jeu de taquin equivalent
  increasing tableaux of type B, then
  $S^2$ and $T^2$ are jeu de taquin equivalent tableaux of type
  A.\smallskip

  \noin{\rm(b)} If\, $U$ is an increasing tableau of straight shape of
  type B such that $U^2$ is a unique rectification target of type A,
  then $U$ is a unique rectification target of type B.
\end{prop}

Notice that an increasing tableau $T$ of type B is a minimal
increasing tableau if and only if $T^2$ is a minimal increasing
tableau of type A.  We can therefore use
Proposition~\ref{prop:translateAB} to derive the following two
corollaries from Theorem~\ref{thm:mininc_urt_A} and
Theorem~\ref{thm:mininc_rect_A}.

\begin{cor}\label{cor:mininc_urt_B}
  Let $\la$ be a shifted Young diagram.  Then the minimal increasing
  tableau $M_\la$ of type B is a unique rectification target.
  Furthermore, an increasing tableau $T$ of type B rectifies to
  $M_\la$ if and only if $w(T^2) = w(M_\la^2)$.
\end{cor}

\begin{cor}\label{cor:mininc_rect_B}
  Let $T$ be any minimal increasing tableau of type B.  Then
  $\rect(T)$ is a minimal increasing tableau.
\end{cor}

We also recover the result from \cite{clifford.thomas.ea:k-theoretic}
that row-wise superstandard tableaux of type B are unique
rectification targets.  Recall that this is not always true for
column-wise superstandard tableaux of type B by
Example~\ref{ex:non_urt_B}.

\begin{cor}[Clifford, Thomas, Yong]
  Let $\la$ be a shifted Young diagram.  Then the row-wise
  superstandard tableau $S_\la$ is a unique rectification target of
  type B.
\end{cor}
\begin{proof}
  This follows because $S_\la^2$ is a unique rectification target by
  Theorem~\ref{thm:fathook}.
\end{proof}

We believe that Proposition~\ref{prop:translateAB}(a) is true in both
directions.  In fact, this follows from Conjecture~\ref{conj:kknuthAB}
below.  However, the following example shows that not all unique
rectification targets of type B can be detected by
Proposition~\ref{prop:translateAB}(b).

\begin{example}
  One can check that the following three increasing tableaux of type A
  are jeu de taquin equivalent:
  \[
  T_1 = \tableau{10}{1&2&4\\ 2&3&5\\ 4&5}
  \ \ \ ; \ \ \ 
  T_2 = \tableau{10}{1&2&4\\ 2&3&5\\ 4}
  \ \ \ ; \ \ \ 
  T_3 = \tableau{10}{1&2&4\\ 2&3\\ 4&5} \ .
  \]
  It follows from Theorem~\ref{thm:mininc_urt_A} and
  Corollary~\ref{cor:lis} that $T_1|_{\{1,2,3,4\}}$ is a unique
  rectification target of type A, from which we deduce that $T_1$,
  $T_2$, and $T_3$ are the only tableaux of straight shapes in their
  jeu de taquin class $[T_1]$.
  
  It follows from this that the increasing tableau of type B defined
  by
  \[
  U = \tableau{10}{1&2&4\\ &3&5}
  \]
  is a unique rectification target.  In fact we have $U^2 = T_1$, so
  if $S$ is any tableau of straight shape of type B that is jeu de
  taquin equivalent to $U$, then $S^2 \in [T_1]$.  Since $S^2$ is a
  tableau of straight shape that is symmetric across the diagonal, we
  deduce that $S^2 = T_1$ and $S=U$.
\end{example}

\subsection{Weak $K$-Knuth equivalence}\label{sec:weak_kknuth}

Jeu de taquin equivalence of increasing tab\-leaux of type A is
governed by $K$-Knuth equivalence of their reading words.  The
following relation plays the analogous role for increasing tableaux of
type B.

\begin{defn}\label{defn:weak_kknuth}
  Define the {\em weak $K$-Knuth equivalence relation\/} on words of
  integers, denoted by $\equivB$, as the symmetric transitive closure
  of the following basic relations.  For any words of integers $\bu$
  and $\bv$, and integers $a,b,c$ we set
  \[
  \begin{split}
    (\bu,a,a,\bv) &\equivB (\bu,a,\bv) \,, \\
    (\bu,a,b,a,\bv) &\equivB (\bu,b,a,b,\bv) \,, \\
    (\bu,a,b,c,\bv) &\equivB (\bu,a,c,b,\bv) \ \ \ \text{if $b < a <
      c$\,,} \\
    (\bu,a,b,c,\bv) &\equivB (\bu,b,a,c,\bv) \ \ \ \text{if $a < c <
      b$\,,} \ \ and \\
    (a,b,\bu) &\equivB (b,a,\bu) \,.
  \end{split}
  \]
\end{defn}

Compared to the $K$-Knuth relation of Definition~\ref{defn:kknuth},
the weak $K$-Knuth relation allows the first two integers in any word
to be interchanged.  Two weakly increasing tableaux $T_1$ and $T_2$ of
type B are weakly $K$-Knuth equivalent, written $T_1 \equivB T_2$, if
a reading word of $T_1$ is weakly $K$-Knuth equivalent to a reading
word of $T_2$.  By Lemma~\ref{lem:tabclass} this defines an
equivalence relation on the set of all weakly increasing tableaux of
type B that admit reading words.

\begin{lemma}\label{lem:resolve_B}
  Let $T$ be a dotted increasing tableau of type B.  Then all
  resolutions of $T$ are weakly $K$-Knuth equivalent.
\end{lemma}
\begin{proof}
  Let $T'$ and $T''$ be resolutions of $T$.  By
  Theorem~\ref{thm:resolve} and Lemma~\ref{lem:interval}(b) we may
  assume that $T'$ and $T''$ agree except for a single dotted box $\al$
  which lies on the diagonal and is surrounded by five boxes
  satisfying $a < b \leq y \leq c < d$:
  \[
  \tableau{14}{a&b&y\\ &\bullet&c\\ &&d}
  \]
  We can assume that $T'(\al)=b$ and $T''(\al)=c$.  If $T$ contains a
  dot in the same row as $y$, then we may further assume that $T'$ and
  $T''$ replace this dot with the maximum of the boxes above and to
  the left of the dot.  And if $T$ contains a box in the same column
  as $y$ but strictly north of $y$, then we may assume that $T'$ and
  $T''$ replace this dot with the minimum of the boxes below and to
  the right of the dot.  In this situation we must have $b \leq y <
  c$, and we can choose reading words for $T'$ and $T''$ which agree
  except that the word of $T'$ starts with $(b,a,b,d,c,y)$ while the
  word of $T''$ starts with $(c,a,b,d,c,y)$.  If $y=b$ then set
  $y'=c$, otherwise set $y'=y$.  Then we have
  \[
  \begin{split}
    (b,a,b,d,c,y) 
    &\equivB (a,b,b,d,c,y)
    \equiv (a,b,d,c,y)
    \equiv (d,c,a,b,y')\\
    &\equiv (d,c,c,a,b,y')
    \equivB (c,d,c,a,b,y')
    \equiv (c,a,b,d,c,y) \,,
  \end{split}
  \]
  as required.
\end{proof}

\begin{thm}\label{thm:jdt_weak_kknuth}
  Let $T$ and $T'$ be increasing tableaux of type B.  Then the reading
  words of $T$ and $T'$ are weakly $K$-Knuth equivalent if and only if
  $T$ and $T'$ are jeu de taquin equivalent for $\Delta$.
\end{thm}
\begin{proof}
  This follows from the same argument as proves
  Theorem~\ref{thm:jdt_kknuth}, except that Lemma~\ref{lem:resolve_A}
  is replaced with Lemma~\ref{lem:resolve_B}.
\end{proof}

\begin{cor}
  Let $T$ be an increasing tableau of straight shape of type B.  The
  following are equivalent.  {\rm(a)} $T$ is a unique rectification
  target for $\Delta$.  {\rm(b)} If $T'$ is any increasing tableau of
  straight shape such that $T' \equivB T$, then $T' = T$.
\end{cor}
\begin{proof}
  This follows from from Theorem~\ref{thm:jdt_weak_kknuth} and
  Lemma~\ref{lem:urt_class}.
\end{proof}

Given a word of integers $\bu = (u_1,u_2,\dots,u_k)$, we let
$\bu^\dagger$ denote this word in reverse order, and we let
$\bu^\dagger \bu = (u_k,\dots,u_2,u_1,u_1,u_2,\dots,u_k)$ denote the
composition of the two words.  It follows from the definitions that,
if $\bu$ and $\bv$ are weakly $K$-Knuth equivalent words, then
$\bu^\dagger \bu$ and $\bv^\dagger \bv$ are $K$-Knuth equivalent.
Computer experiments suggest that the converse is also true.

\begin{conj}\label{conj:kknuthAB}
  Let $\bu$ and $\bv$ be words of integers.  Then we have $\bu \equivB
  \bv$ if and only if $\bu^\dagger \bu \equiv \bv^\dagger \bv$.
\end{conj}

If this conjecture is true, then the converse of
Proposition~\ref{prop:translateAB}(a) holds.  Furthermore, if $U$ is
any increasing tableau of straight shape of type B, then
Conjecture~\ref{conj:kknuthAB} implies that $U$ is a unique
rectification target for $\Delta$ if and only if $U^2$ is the only
tableau in its jeu de taquin class $[U^2]_{\N^2}$ that has straight
shape and is {\em symmetric\/} across the diagonal.

\subsection{The Pieri rule for $\Gamma(\Delta)$}

Define a {\em Pieri word of type B\/} to be a sequence of integers $a
= (a_1, a_2, \dots, a_\ell)$ such that each entry $a_i$ is either
smaller than or equal to all predecessors $a_j$ with $j<i$, or larger
than or equal to all predecessors.  We also define a {\em Pieri
  tableau of type B} to be any increasing tableau $T$ of type B such
that the row word of $T$ is a Pieri word of type B.  Such a tableau is
called a {\em KOG-tableau\/} in \cite{buch.ravikumar:pieri}.  Finally,
define the {\em range\/} of a tableau to be the set of values
contained in its boxes.  The following tableau is a Pieri tableau of
type B with range $[1,6] = \{1,2,3,4,5,6\}$.
\[
\tableau{10}{
  [lh]&[t]&{}&{}&{}&{}&{}&[a]1&6\\
  &[lb]&&&&&&[a]5\\
  &&[lb]&&&[a]2&5\\
  &&&2&3&4
}\medskip
\]
The following result shows that multiplication with single-row shapes
in the combinatorial $K$-theory ring $\Gamma(\Delta)$ is compatible
with the Pieri formula for maximal orthogonal Grassmannians proved in
\cite{buch.ravikumar:pieri}.

\begin{cor}\label{cor:pieri_B}
  Let $\la$ be a shifted Young diagram and let $p$ be a positive
  integer.  Then we have
  \[
  G_{(p)} \cdot G_\la = \sum_\nu c^\nu_{p,\la}\, G_\nu
  \]
  in the ring $\Gamma(\Delta)$, where the sum is over all shifted
  Young diagrams $\nu \supset \la$, and $c^\nu_{p,\la}$ is the number
  of Pieri tableaux of type B with shape $\nu/\la$ and range $[1,p]$.
\end{cor}
\begin{proof}
  An inspection of the basic relations of
  Definition~\ref{defn:weak_kknuth} shows that, if $\bu$ and $\bv$ are
  weakly $K$-Knuth equivalent words, then $\bu$ is a Pieri word of
  type B if and only if $\bv$ is a Pieri word of type B.
  Theorem~\ref{thm:jdt_weak_kknuth} therefore implies that the set $P$
  of all Pieri tableaux of type B with range $[1,p]$ is closed under
  jeu de taquin slides.  Since $M_{(p)}$ is the only tableau of
  straight shape in $P$, we deduce that $P$ is equal to the jeu de
  taquin class $[M_{(p)}]_\Delta$.  The corollary follows from this.
\end{proof}

\begin{remark}
  For each permutation $w \in \Sigma$ we can define an element $B_w =
  \sum_\la b_{w,\la}\, G_\la \in \Gamma(\Delta)$, where $b_{w,\la}$ is
  the number of increasing tableaux $T$ of shape $\la$ such that
  $w(T^2) = w$.  Notice that $B_w \neq 0$ if and only if $w^{-1} = w$.
  These elements also satisfy the identity $B_w \cdot G_\la = \sum_\nu
  c^\nu_{w,\la}\, G_\nu$, where $c^\nu_{w,\la}$ is the number of
  increasing tableaux $T$ of shape $\nu/\la$ such that $w(T^2) = w$,
  which is analogous to Corollary~\ref{cor:sgroth_prod}.  It is
  interesting to ask if the elements $B_w$ have any geometric meaning
  in the $K$-theory of a maximal orthogonal Grassmannian.
\end{remark}



\begin{thebibliography}{10}

\bibitem{billey.jockusch.ea:some}
S.~C. Billey, W.~Jockusch, and R.~P. Stanley, \emph{Some combinatorial
  properties of {S}chubert polynomials}, J. Algebraic Combin. \textbf{2}
  (1993), no.~4, 345--374. \MR{1241505 (94m:05197)}

\bibitem{brion:positivity}
M.~Brion, \emph{Positivity in the {G}rothendieck group of complex flag
  varieties}, J. Algebra \textbf{258} (2002), no.~1, 137--159, Special issue in
  celebration of Claudio Procesi's 60th birthday. \MR{1958901 (2003m:14017)}

\bibitem{buch:littlewood-richardson}
A.~S. Buch, \emph{A {L}ittlewood-{R}ichardson rule for the {$K$}-theory of
  {G}rassmannians}, Acta Math. \textbf{189} (2002), no.~1, 37--78. \MR{1946917
  (2003j:14062)}

\bibitem{buch:combinatorial}
\bysame, \emph{Combinatorial {$K$}-theory}, Topics in cohomological studies of
  algebraic varieties, Trends Math., Birkh\"auser, Basel, 2005, pp.~87--103.
  \MR{2143073 (2007a:14056)}

\bibitem{buch.kresch.ea:stable}
A.~S. Buch, A.~Kresch, M.~Shimozono, H.~Tamvakis, and A.~Yong, \emph{Stable
  {G}rothendieck polynomials and {$K$}-theoretic factor sequences}, Math. Ann.
  \textbf{340} (2008), no.~2, 359--382. \MR{2368984 (2009c:05250)}

\bibitem{buch.mihalcea:curve}
A.~S. Buch and L.~C. Mihalcea, \emph{Curve neighborhoods of {S}chubert
  varieties}, preprint, 2013.

\bibitem{buch.ravikumar:pieri}
A.~S. Buch and V.~Ravikumar, \emph{Pieri rules for the {$K$}-theory of
  cominuscule {G}rassmannians}, J. Reine Angew. Math. \textbf{668} (2012),
  109--132. \MR{2948873}

\bibitem{chaput.manivel.ea:quantum*1}
P.-E. Chaput, L.~Manivel, and N.~Perrin, \emph{Quantum cohomology of minuscule
  homogeneous spaces}, Transform. Groups \textbf{13} (2008), no.~1, 47--89.
  \MR{2421317 (2009e:14095)}

\bibitem{clifford.thomas.ea:k-theoretic}
E.~Clifford, H.~Thomas, and A.~Yong, \emph{{$K$}-theoretic {S}chubert calculus
  for {OG}(n,2n+1) and jeu de taquin for shifted increasing tableaux},
  arXiv:1002.1664.

\bibitem{fan:hecke}
C.~K. Fan, \emph{A {H}ecke algebra quotient and some combinatorial
  applications}, J. Algebraic Combin. \textbf{5} (1996), no.~3, 175--189.
  \MR{1394304 (97k:20013)}

\bibitem{fomin.kirillov:grothendieck}
S.~Fomin and A.~N. Kirillov, \emph{Grothendieck polynomials and the
  {Y}ang-{B}axter equation}, Proc. Formal Power Series and Alg. Comb. (1994),
  183--190.

\bibitem{fulton:intersection}
  W.~Fulton, \emph{Intersection theory}, second ed., Ergebnisse der
  Mathematik und ihrer Grenzgebiete (3), vol.~2, Springer-Verlag,
  Berlin, 1998. \MR{MR1644323 (99d:14003)}

\bibitem{graham:equivariant}
W.~Graham, \emph{Equivariant {$K$}-theory and {S}chubert varieties}, preprint,
  2002.

\bibitem{humphreys:introduction}
J.~E. Humphreys, \emph{Introduction to {L}ie algebras and representation
  theory}, Springer-Verlag, New York, 1972, Graduate Texts in Mathematics, Vol.
  9. \MR{0323842 (48 \#2197)}

\bibitem{humphreys:reflection}
\bysame, \emph{Reflection groups and {C}oxeter groups}, Cambridge Studies in
  Advanced Mathematics, vol.~29, Cambridge University Press, Cambridge, 1990.
  \MR{1066460 (92h:20002)}

\bibitem{knutson:schubert}
A.~Knutson, \emph{Schubert patches degenerate to subword complexes}, Transform.
  Groups \textbf{13} (2008), no.~3-4, 715--726. \MR{2452612 (2009g:14062)}

\bibitem{kostant.kumar:t-equivariant}
B.~Kostant and S.~Kumar, \emph{{$T$}-equivariant {$K$}-theory of generalized
  flag varieties}, J. Differential Geom. \textbf{32} (1990), no.~2, 549--603.
  \MR{1072919 (92c:19006)}

\bibitem{lascoux:transition}
A.~Lascoux, \emph{Transition on {G}rothendieck polynomials}, Physics and
  combinatorics, 2000 ({N}agoya), World Sci. Publ., River Edge, NJ, 2001,
  pp.~164--179. \MR{1872255 (2002k:14082)}

\bibitem{lascoux.schutzenberger:structure}
A.~Lascoux and M.-P. Sch{\"u}tzenberger, \emph{Structure de {H}opf de l'anneau
  de cohomologie et de l'anneau de {G}rothendieck d'une vari\'et\'e de
  drapeaux}, C. R. Acad. Sci. Paris S\'er. I Math. \textbf{295} (1982), no.~11,
  629--633. \MR{686357 (84b:14030)}

\bibitem{lenart:combinatorial}
C.~Lenart, \emph{Combinatorial aspects of the {$K$}-theory of {G}rassmannians},
  Ann. Comb. \textbf{4} (2000), no.~1, 67--82. \MR{1763950 (2001j:05124)}

\bibitem{lenart.postnikov:affine}
C.~Lenart and A.~Postnikov, \emph{Affine {W}eyl groups in {$K$}-theory and
  representation theory}, Int. Math. Res. Not. IMRN (2007), no.~12, Art. ID
  rnm038, 65. \MR{2344548 (2008j:14105)}

\bibitem{perrin:small*1}
N.~Perrin, \emph{Small resolutions of minuscule {S}chubert varieties}, Compos.
  Math. \textbf{143} (2007), no.~5, 1255--1312. \MR{2360316 (2008m:14098)}

\bibitem{proctor:d-complete}
R.~Proctor, \emph{$d$-{C}omplete {P}osets {G}eneralize {Y}oung {D}iagrams for
  the {J}eu de {T}aquin {P}roperty}, http://www.math.unc.edu/Faculty/rap/,
  2004.

\bibitem{proctor:bruhat}
R.~A. Proctor, \emph{Bruhat lattices, plane partition generating functions, and
  minuscule representations}, European J. Combin. \textbf{5} (1984), no.~4,
  331--350. \MR{782055 (86h:17007)}

\bibitem{richardson:intersections}
R.~W. Richardson, \emph{Intersections of double cosets in algebraic groups},
  Indag. Math. (N.S.) \textbf{3} (1992), no.~1, 69--77. \MR{1157520
  (93b:20081)}

\bibitem{stembridge:fully}
J.~R. Stembridge, \emph{On the fully commutative elements of {C}oxeter groups},
  J. Algebraic Combin. \textbf{5} (1996), no.~4, 353--385. \MR{1406459
  (97g:20046)}

\bibitem{thomas.yong:combinatorial}
H.~Thomas and A.~Yong, \emph{A combinatorial rule for (co)minuscule {S}chubert
  calculus}, Adv. Math. \textbf{222} (2009), no.~2, 596--620. \MR{2538022
  (2011b:14122)}

\bibitem{thomas.yong:jeu}
\bysame, \emph{A jeu de taquin theory for increasing tableaux, with
  applications to {$K$}-theoretic {S}chubert calculus}, Algebra Number Theory
  \textbf{3} (2009), no.~2, 121--148. \MR{2491941 (2010h:05334)}

\bibitem{thomas.yong:direct}
\bysame, \emph{The direct sum map on {G}rassmannians and jeu de taquin for
  increasing tableaux}, Int. Math. Res. Not. IMRN (2011), no.~12, 2766--2793.
  \MR{2806593 (2012f:14097)}

\bibitem{willems:k-theorie}
M.~Willems, \emph{{$K$}-th\'eorie \'equivariante des tours de {B}ott.
  {A}pplication \`a la structure multiplicative de la {$K$}-th\'eorie
  \'equivariante des vari\'et\'es de drapeaux}, Duke Math. J. \textbf{132}
  (2006), no.~2, 271--309. \MR{2219259 (2007b:19009)}

\bibitem{worley:theory}
D.~Worley, \emph{A theory of shifted {Y}oung tableaux}, Ph.D. thesis,
  Massachusetts Institute of Technology, 1984.

\end{thebibliography}

\providecommand{\bysame}{\leavevmode\hbox to3em{\hrulefill}\thinspace}
\providecommand{\MR}{\relax\ifhmode\unskip\space\fi MR }
\providecommand{\MRhref}[2]{%
  \href{http://www.ams.org/mathscinet-getitem?mr=#1}{#2}
}
\providecommand{\href}[2]{#2}

\end{document}